\documentclass[10pt]{article}

\usepackage{graphicx}
\usepackage{amsmath}
\usepackage{amsfonts}
\usepackage{amssymb}
\usepackage{enumerate}
\usepackage{lscape}
\usepackage{amsthm}
\usepackage{longtable}
\usepackage{rotating}
\usepackage{multirow}
\usepackage{color}
\usepackage{url}
\usepackage{caption}
\usepackage{subcaption}
\usepackage{rotating}
\usepackage{algorithm}
\usepackage{algorithmic}
\usepackage{graphicx}

\newtheorem{definition}{Definition}[section]

\newtheorem{proposition}{Proposition}[section]

\newtheorem{lemma}{Lemma}[section]

\newtheorem{theorem}{Theorem}[section]

\newtheorem{assumption}{Assumption}[section]

\newcommand{\Sc}[2]{\langle#1, #2\rangle}

\newcommand{\sm}{\setminus}
\newcommand{\tx}[1]{\textnormal{#1}}

\newcommand{\PP}{\ensuremath{\mathsf P}}

\newcommand{\R}{\ensuremath \mathbb{R}}

\newcommand{\n}[1]{\|#1\|}

\newcommand{\MM}{\ensuremath \mathcal{M}}

\newcommand{\LG}{L_{\Gamma}}

\newcommand{\td}[1]{\tilde{#1}}
\newcommand{\Lr}{L_r}
\newcommand{\Lf}{L_f}
\newcommand{\grad}{\tx{grad}}
\newcommand{\tP}{\Gamma}
\newcommand{\Sb}{\mathbb{S}}
\usepackage{lipsum}
\usepackage{amsfonts}
\usepackage{graphicx}
\usepackage{epstopdf}
\usepackage{algorithmic}
\ifpdf
\DeclareGraphicsExtensions{.eps,.pdf,.png,.jpg}
\else
\DeclareGraphicsExtensions{.eps}
\fi

\DeclareMathOperator{\dist}{dist}

\title{Retraction based Direct Search Methods for Derivative Free Riemannian Optimization}

\author{ Vyacheslav Kungurtsev\thanks{ Department of Computer Science,
		Czech Technical University, Czech Republic. Research supported by the OP VVV project CZ.02.1.01/0.0/0.0/16\_019/0000765 ``Research Center for Informatics" 
		(\tt{kunguvya@fel.cvut.cz}) 
}
	\and
	Francesco~Rinaldi\thanks{Dipartimento di Matematica ``Tullio Levi-Civita'', Universit\`a
		di Padova, Italy
		(\tt{rinaldi@math.unipd.it})}
	\and 
	Damiano~Zeffiro\thanks{Dipartimento di Matematica ``Tullio Levi-Civita'', Universit\`a
		di Padova, Italy
		(\tt{damiano.zeffiro@math.unipd.it})}
}

\begin{document}
	\maketitle
\begin{abstract}
    Direct search methods represent a robust and reliable class of algorithms for solving black-box optimization problems. 
    In this paper, we  explore the application of those strategies to Riemannian optimization, wherein minimization is to be performed with respect to variables restricted to lie on a manifold. More specifically, we consider classic and line search extrapolated variants of direct search, and, by making use of retractions, we devise tailored strategies for the minimization of both smooth and nonsmooth functions.
     As such we analyze, for the first time in the literature, a class of retraction based algorithms for minimizing nonsmooth objectives on a Riemannian manifold without having access to (sub)derivatives. Along with convergence guarantees we provide a set of numerical performance illustrations on a standard set of problems. \\
     	\textbf{Keywords:} Direct search, derivative free optimization, Riemannian manifold, retraction. \\
     	\textbf{AMS subject classifications:} 90C06, 90C30, 90C56. 
\end{abstract}

	\section{Introduction}
    Riemannian optimization, or solving minimization problems constrained on a Riemannian manifold embedded in an Euclidean space, is an important and active area of research considering the numerous problems in data science, robotics, and other settings wherein there is an important geometric structure characterizing the allowable inputs. Derivative Free Optimization (DFO), or Zeroth Order Optimization, involves algorithms that only make use of function evaluations rather than any gradient computations in their implementation. In cases of dynamics subject to significant epistemic uncertainty and the necessity of performing a simulation to compute a function evaluation, derivatives may be unavailable. This paper presents the introduction of a classic set of DFO algorithms, namely direct search, to the case of Riemannian optimization. For classic references of Riemannian optimization and DFO, see, e.g., ~\cite{absil2009optimization} and \cite{audet2017derivative,conn2009introduction,larson2019derivative}, respectively.
	
	Formally, let $\MM$ be a smooth manifold embedded in $\R^n$.
	We are interested here in the problem 
	\begin{equation} \label{eq:opt}
		\min_{x \in \MM} f(x)
	\end{equation}
	with $f$ continuous and bounded below. We consider both the case of $f(x)$ being continuously differentiable, as well as the more general nonsmooth case.
	
	Direct search methods (see, e.g.,~\cite{kolda2003optimization} and references therein) belong to the class of algorithms that are mesh based, rather than model based. This distinction presents a binary taxonomy of DFO algorithms: on the one hand we have those based on approximating gradient information using function evaluations and constructing approximate local models, while on the other hand we have those based on sampling a pre-defined grid of points for the next iteration. Thus direct search is particularly suitable for black box cases wherein it is unknown the degree to which any model would have much veracity. 
	
	To the best of our knowledge, thorough studies of DFO on Riemannian manifolds have only been carried out recently in the literature. In~\cite{li2020zeroth}, the authors focus on a model based method using a two point function approximation for the gradient. 
	The paper~\cite{yao2021riemannian} presents a specialized Polak-Ribi{\' e}ere-Polyak procedure for finding a zero of a tangent vector field on a Riemannian manifold. In \cite{dreisigmeyer2018direct}, the author focuses on a specific class of manifolds (reductive homogeneous spaces, including several matrix manifolds) where, thanks to the properties of exponential maps, a straightforward extension of mesh adaptive direct search methods (see, e.g., \cite{audet2006mesh,audet2017derivative}) and probabilistic direct search strategies \cite{gratton2015direct} is possible. Some DFO methods and nonsmooth problems on Riemannian manifolds without convergence analysis can be found in \cite{hosseini2019nonsmooth} and references therein. 
	
	Thus our paper presents the first analysis of  retraction based direct search strategies on Riemannian manifolds, and the first analysis of a DFO algorithm for minimizing nonsmooth objectives in Riemannian optimization. In particular, we first adapt, thanks to the use of retractions, a classic direct search scheme (see, e.g., \cite{conn2009introduction, kolda2003optimization}) and a linesearch based scheme (see, e.g., \cite{cristofari2021derivative,liuzzi2010sequential,lucidi2002derivative,lucidi2002global} for further details on this class of methods) 
	to deal with the minimization of a given smooth function over a manifold. Then, inspired by the ideas in \cite{fasano2014linesearch}, we extend the two proposed strategies to the nonsmooth case.  
	
	The remainder of this paper is as follows. In Section~\ref{s:def}, we present some definitions. In Section~\ref{s:smooth}, we present and prove convergence for a direct search method applicable for continuously differentiable $f$. In Section~\ref{s:nonsmooth}, we consider the case of $f$ not being continuously differentiable, and only Lipschitz continuous. We present some numerical results in Section~\ref{s:num} and conclude in Section~\ref{s:con}.

	\section{Definitions and notation}\label{s:def}
		We now introduce some notation for the formalism we use in this article. We refer the reader to, e.g., \cite{absil2009optimization,boumal2020introduction} for an overview of the relevant background. \\
	Let
	$T\MM$ be the tangent manifold and for $x \in \MM$ let $T_x\MM$ be the tangent bundle to $\MM$ in~$x$. We assume that $\MM$ is a Riemannian manifold, i.e., for $x$ in $\MM$, we have a scalar product $\Sc{\cdot}{\cdot}_x : T_x\MM \times T_x\MM \rightarrow \R$ smoothly dependent from $x$. Let $\dist(\cdot, \cdot)$ be the distance induced by the scalar product, so that for $x, y \in \MM$ we have that $\dist(x, y)$ is the length of the shortest geodesic connecting $x$ and $y$.  Furthermore, let $\nabla_{\MM}$ be the Levi-Cita connection for $\MM$, and $\Gamma: T\MM \times \MM \rightarrow T\MM$ the parallel transport with respect to $\nabla_{\MM}$, with $\Gamma_x^y(v) \in T_y\MM$ transport of the vector $v \in T_x\MM$. We define $\PP_x$ as the orthogonal projection from $\R^n$ to $T_x\MM$, and $S(x, r) \subset \R^n$ as the sphere centered at $x$ and with radius $r$.  \\
	We write $\{a_k\}$ as a shorthand for $\{a_k\}_{k \in I}$ when the index set $I$ is clear from the context. We also use the shorthand notations $T_k\MM, \PP_k, \Sc{\cdot}{\cdot}_k, \n{\cdot}_k$, $\Gamma_i^j$ for $T_{x_k}\MM, \PP_{x_k}, \Sc{\cdot}{\cdot}_{x_k}, \n{\cdot}_{x_k}$ and $\Gamma_{x_i}^{x_j}$. \\ 
	We define the distance $\dist^*$ between vectors in different tangent spaces in a standard way using parallel transport (see for instance \cite{azagra2005nonsmooth}): for $x, y \in M$, $v \in T_xM$ and $w \in T_yM$, 
	\begin{equation} \label{eq:ct}
		\dist^*(v, w) = \n{v - \Gamma_y^x w} = \n{w - \Gamma_x^y v}  \, ,
	\end{equation}
	and for a sequence $\{(y_k, v_k)\}$  in $T\MM$ we write $v_k \rightarrow v$ if $y_k \rightarrow y$ in $\MM$ and $\dist^*(v_k, v) \rightarrow 0$. \\
	As it is common in the Riemannian optimization literature (see, e.g., \cite{absil2012projection}), to define our tentative descent directions we use a retraction $R: T\MM \rightarrow \MM$. We assume $R \in C^1(T\MM, \MM)$, with
	\begin{equation} \label{eq:rbounded}
		\dist (R(x, d), x) \leq \Lr \n{d} \, ,
	\end{equation}
	(true in any compact subset of $T\MM$ given the $C^1$ regularity of $R$, without any further assumptions), and that the sufficient decrease property holds: for any $L$-Lipschitz smooth $f$,
	\begin{equation} \label{eq:taylor}
		f(R(x, d)) \leq f(x) + \Sc{\tx{\grad} f(x)}{d} + L \n{d}^2 \, .
	\end{equation}
	
	\section{Smooth optimization problems}\label{s:smooth}
	In this section, we consider solving~\eqref{eq:opt} with the objective satisfying $f \in C^1(\MM)$. Recall that we can define the Riemannian gradient as
	\begin{equation}
		\grad f(x) = \PP_x(\nabla f(x)),
	\end{equation} 
	for given $x \in \MM$.
	\subsection{Preliminaries}
	
	First, we assume that the objective function $f$ has a Lipschitz continuous gradient on the manifold.
	\begin{assumption}\label{as:lipcongrad}
	There exists $L_f>0$ such that for all $x\in\MM$
	\begin{equation} \label{eq:lipf}
		\dist^*(\grad f(x), \grad f(y)) = \n{\tP_x^y \grad f(x) - \grad f(y)} \leq \Lf \n{\grad f(x)} \, ,
	\end{equation}
	\end{assumption}
Like in the unconstrained case, the Lipschitz gradient property implies the standard descent property.  	
\begin{proposition} \label{p:std}
	Assume that $M$ is compact and $R$ is a $C^2$ retraction. If condition \eqref{eq:lipf} holds, then the sufficient decrease property \eqref{eq:taylor} holds for some constant $L > 0$.
\end{proposition}
The proof can be found in the appendix. An analogous property, but under the stronger assumption that $f$ has Lipschitz gradient as a function in $\R^n$, is proved in \cite{boumal2019global}.\\
Another assumption we make in this context  is that the gradient norm is globally bounded.
	\begin{assumption}\label{as:globboundgrad}
	There exists $M_f>0$ such that,
	\begin{equation} \label{eq:mf}
		\n{\grad f(x)} \leq M_f,
	\end{equation}
	for every $x \in \MM$. \end{assumption}
	
	For each of the algorithms in this section, we further assume that, at each iteration $k$, we have a positive spanning basis $\{p_k^j\}_{j \in [1:K]}$ of the tangent space $T_{x_{k}}M$ of the iterate $x_k$ (further details on how to get a positive spanning basis can be found, e.g., in \cite{conn2009introduction}). 	More specifically, we assume that the basis stays bounded and does not become degenerate during the algorithm, that is, 
	\begin{assumption}\label{as:basis}
	There exists $B>0$ such that
	\begin{equation} \label{eq:bounded}
		\max_{j \in [1:K]} \n{p_k^j} \leq B,
	\end{equation}
	for  every $k \in \mathbb{N}$. Furthermore there is a constant $\tau > 0$ such that
	\begin{equation} \label{eq:pspanningt}
		\max_{i \in [1:K]} \Sc{r}{p_k^j} \geq \tau \n{r}, 
	\end{equation}
	for every $k \in \mathbb{N}$ and $r \in T_{x_k}M$. \\
	\end{assumption}
	\subsection{Direct search algorithm}
		We present here our Riemannian Direct Search method based on Spanning Bases (RDS-SB) for smooth objectives as Algorithm~\ref{alg:ds}.	
	\begin{algorithm}[h]
		\caption{RDS-SB}
		\label{alg:ds}
		\begin{algorithmic}
			\par\vspace*{0.1cm}
			\STATE \textbf{Input:} $x_0\in \MM$, $\gamma_1 \in (0, 1)$, $\gamma_2 \geq 1$, $\alpha_0 > 0$, $\rho > 0$
			\FOR{$k = 0, 1, ...$}
		    \STATE Compute a positive spanning basis $\{p_k^j\}_{j = 1:K}$ of $T_k\MM$
			\FOR{$j=1,..., K$}
			\STATE Let $x_k^j= R(x_k, \alpha_k p_k^j)$
			\IF{$ f(x_k^j) \le f(x_k)-\rho \alpha_k^2 $}
			\STATE $\alpha_{k + 1} = \gamma_2 \alpha_k, x_{k + 1} = x_k^j$
			\STATE Declare the step $k$ successful 
			\STATE \textbf{Break}    
			\ENDIF
			\ENDFOR
			\IF{$ f(x_k^j) > f(x_k)-\rho \alpha_k^2 $ for $j \in [1:K]$}
			\STATE $\alpha_{k+1}= \gamma_1 \alpha_k$, $x_{k + 1} = x_k$
			\STATE Declare the step $k$ unsuccessful
			\ENDIF
			\ENDFOR
		\end{algorithmic}
	\end{algorithm}
	
	This procedure resembles the standard direct search algorithm for unconstrained derivative free optimization (see, e.g.,~\cite{conn2009introduction, kolda2003optimization}) with two significant modifications. First, at every iteration a positive spanning basis is computed for the current tangent vector space $T_k\MM$. As this space is expected to change at every iteration, it is not possible to use the same standard positive spanning sets appearing in the classic algorithms. Second, the candidate point $x_k^j$ is computed by retracting the step $\alpha_k p_k^j$ from the current tangent space $T_{x_k^j}\MM$ to the manifold.

	\subsection{Convergence analysis}
	Now we show asymptotic global convergence of the method. Using a similar structure of reasoning as in standard convergence derivations for direct search, we prove that the gradient evaluated at iterates associated with unsuccessful steps must converge to zero, and extend the property to the remaining iterates, using the Lipschitz continuity of the gradient.
	
	The first lemma states a bound on the scalar product between the gradient and the descent direction  for an unsuccessful iteration.
	\begin{lemma} \label{l:deltaklem}
		If $f(R(x_k, \alpha_k p_k^j)) > f(x_k) - \gamma \alpha_k^2$, then
		\begin{equation} \label{eq:deltak}
			\alpha_k(LB^2 + \gamma) > - \Sc{\tx{\grad} f(x_k)}{p_k^j} \, .
		\end{equation}
	\end{lemma}
	\begin{proof}
		To start with, we have
		\begin{equation} \label{eq:pspanning}
			\begin{aligned}
				&	f(x_k) - \gamma \alpha_k^2 < f(R(x, \alpha_k p_k^j)) \leq f(x_k) + \alpha_k\Sc{\tx{\grad} f(x_k)}{p_k^j} + L \alpha_k^2 \n{p_k^j}^2 \\ 
				\leq &  f(x_k) + \alpha_k \Sc{\tx{\grad} f(x_k)}{p_k^j} + L \alpha_k^2 B^2 \, ,
			\end{aligned}
		\end{equation}
		where we used \eqref{eq:taylor} in the second inequality, and \eqref{eq:bounded} in the third one.
		The above inequality can be rewritten as
		\begin{equation}
			\alpha_k\Sc{\tx{\grad} f(x_k)}{p_k^j} + \alpha_k^2 (LB^2 + \gamma) > 0.
		\end{equation}
		Given that $\alpha_k > 0$, the above is true iff
		\begin{equation} \label{eq:deltakbis}
			\alpha_k > - \frac{\Sc{\tx{\grad} f(x_k)}{p_k^j}}{(LB^2 + \gamma)} \, ,
		\end{equation}
		which rearranged gives the thesis. 
	\end{proof}

	From this we can infer a bound on the gradient with respect to the stepsize.
	\begin{lemma} \label{l:gnorm}
		If iteration $k$ is unsuccessful, then 
		\begin{equation} \label{eq:ngrad}
			\n{\grad f(x_k)} \leq \frac{\alpha_k(2LB^2 + \gamma)}{\tau} \, .
		\end{equation}
	\end{lemma}
	\begin{proof}
		If iteration $k$ is unsuccessful, equation  \eqref{eq:deltak} must hold for every $j \in [1:K]$. We obtain the thesis by applying the positive spanning property \eqref{eq:pspanningt} in the RHS:
		\begin{equation}
			\alpha_k(LB^2 + \gamma) > \max_{j \in [1:K]} - \Sc{\tx{\grad} f(x_k)}{p_k^j} \geq \tau \n{\grad f(x_k)} \, .
		\end{equation} 
	\end{proof}

    Finally, we are able to show convergence of the gradient norm using the lemmas above and appropriate arguments regarding the step sizes.
	\begin{theorem} \label{th:ds}
		For the sequence $\{x_k\}$ generated by Algorithm \ref{alg:ds} we have
		\begin{equation}
			\lim_{k \rightarrow \infty} \n{\grad f(x_k)} = 0 \, .
		\end{equation}
	\end{theorem}
	\begin{proof}
		To start with, clearly $\alpha_k \rightarrow 0$ since the objective is bounded below, $\{f(x_k)\}$ is non increasing with $f(x_{k + 1}) \leq f(x_k) - \gamma \alpha_k^2$ if the step $k$ is successful, and so there can be a finite number of successful steps with $\alpha_k \geq \varepsilon$ for any $\varepsilon > 0$.\\
		For a fixed $\varepsilon > 0$, let $\bar{k}$ such that $\alpha_k \leq \varepsilon$ for every $k \geq \bar{k}$.  We now show that, for every $\varepsilon > 0$ and $k \geq \bar{k}$ large enough, we have
		\begin{equation} \label{eq:gradbound}
			\n{\grad f(x_k)} \leq  \varepsilon\left(\frac{ (2LB^2 + \gamma)}{\tau} + L_fL_r B  \frac{\gamma_2}{\gamma_2 - 1}\right) \, , 
		\end{equation}
		which clearly implies the thesis given that $\varepsilon$ is arbitrary. \\
		First, \eqref{eq:gradbound} is satisfied for $k \geq \bar{k}$ if the step $k$ is unsuccessful by Lemma \ref{l:gnorm}:
		\begin{equation} \label{eq:unsucc}
			\n{\grad f(x_k)} \leq \frac{\alpha_k(2LB^2 + \gamma)}{\tau} \leq \frac{\varepsilon(2LB^2 + \gamma)}{\tau} \, ,
		\end{equation}
		using $\alpha_k \leq \varepsilon$ in the second inequality. \\
		If the step $k$ is successful, then let $j$ be the minimum positive index such that the step $k + j$ is unsuccessful. We have that $\alpha_{k + i} = \alpha_k \gamma_2^{i}$ for $i \in [0:j - 1]$, and since $\alpha_{k + j - 1} \leq \varepsilon$ by induction we get $\alpha_{k + i} \leq \varepsilon \gamma_2^{i - j + 1}$. Therefore
		\begin{equation} \label{eq:sumdelta}
			\sum_{i= 0}^{j - 1} \alpha_{k + i} \leq \sum_{i = 0}^{j - 1} \varepsilon \gamma_2^{i - j + 1} \leq \varepsilon \sum_{h= 0}^{\infty} \gamma_2^{-h} = \varepsilon \frac{\gamma_2}{\gamma_2 - 1} \, .
		\end{equation}
		Then
		\begin{equation} \label{eq:distbound}
			\begin{aligned}
				& \dist(x_k, x_{k + j}) \leq \sum_{i = 0}^{j - 1} \dist(x_{k + i}, x_{k + i + 1}) = \sum_{i = 0}^{j - 1} \dist(x_{k + i}, R(x_{k + i}, \alpha_{k + i}p_{k + i}^{j(k + i)} )) \\
				\leq & \sum_{i = 0}^{j - 1} L_r \alpha_{k + i}B \leq L_rB \varepsilon \frac{\gamma_2}{\gamma_2 - 1} \, . 	
			\end{aligned}
		\end{equation}
		where we used \eqref{eq:rbounded} together with \eqref{eq:bounded} in the second inequality, and \eqref{eq:sumdelta} in the third one. \\
		In turn,
		\begin{equation}
			\begin{aligned}
				& \n{\grad f(x_k)} \leq \dist^*(\grad f(x_k), \grad f(x_{k + j})) + \n{\grad f(x_{k + j})} \\
				\leq & L_f \dist(x_k, x_{k + j}) + \frac{ \varepsilon(2LB^2 + \gamma)}{\tau} \leq \varepsilon\left(\frac{ 2LB^2 + \gamma}{\tau} + L_fL_r B \frac{\gamma_2}{\gamma_2 - 1} \right) \, ,
			\end{aligned}
		\end{equation}
		where we used \eqref{eq:lipf} and \eqref{eq:unsucc} with $k + j$ instead of $k$ for the first and second summand respectively in the second inequality, and \eqref{eq:distbound} in the last one.
	\end{proof}


	\subsection{Incorporating an extrapolation linesearch}
    The works~\cite{lucidi2002derivative,lucidi2002global} introduced the use of an extrapolating line search that tests the objective on variable inputs farther away from the current iterate than the tentative point obtained by direct search on a given direction (i.e., an element of the positive spanning set). Such a thorough exploration of the search directions ultimately yields better performances in practice.
    We found that the same technique can be applied in the Riemannian setting to good effect.
    We present here our Riemannian Direct Search with Extrapolation method based on Spanning Bases (RDSE-SB) for smooth objectives.	The scheme is presented in detail as Algorithm \ref{alg:dse}.
    As we can easily see, the method uses a specific stepsize for each direction in the positive spanning basis, so that instead of $\alpha_k$ we have a set of stepsizes $\{\alpha_k^j\}_{j \in [1:K]}$ for every $k \in \mathbb{N}_0$. Furthermore 
    a retraction based linesearch procedure
    (see Algorithm \ref{alg:lsep}) is used to better explore a given direction in case a sufficient decrease of the objective is obtained.

	When analyzing the RDSE-SB method, we assume that  the following continuity condition holds.
	\begin{assumption}
	For every $l, m \in \mathbb{N}$, $j \in [1:K]$, there exists a constant $\LG>0$ such that
	
	\begin{equation} \label{eq:mb}
		\dist^*(p^j_l, p_m^j) \leq \LG \dist(x_l, x_m) \, .
	\end{equation}
	\end{assumption}
	We refer the reader to \cite{lucidi2002global} for a slightly weaker continuity condition in an Euclidean setting.

	\begin{algorithm}[h]
		\caption{RDSE-SB}
		\label{alg:dse}
		\begin{algorithmic}
			\par\vspace*{0.1cm}
			\STATE \textbf{Input:} $x_0\in \mathbb{R}^n$, $\{\alpha^j_0\}_{j \in [1:K]}$, $\gamma > 0, \gamma_1 \in (0, 1), \gamma_2 \geq 1$.
		\FOR{  $k=0, 1,...$}
			\STATE Compute a positive spanning basis $\{ p_k^j\}_{j \in [1:K]} $ of $T_k\mathcal{M}$
			\STATE Set $j(k)=\mod(k,n)$, $\alpha_{k}^i = \tilde{\alpha}_{k}^i$ and $\tilde{\alpha}_{k + 1}^i = \tilde{\alpha}_{k}^i$  for $i \in [1:K] \sm \{j(k)\}$. 
			\STATE Compute $\alpha_k^{j(k)}, \tilde{\alpha}_{k + 1}^{j(k)}$ with \textbf{Linesearchprocedure}($\tilde \alpha_k^{j(k)}, x_k, p_k^{j(k)}, \gamma, \gamma_1, \gamma_2$)  
		   \STATE Set $x_{k+1}= R(x_k, \alpha_k^{j(k)} p_k^{j(k)})$	
		   \ENDFOR
		\end{algorithmic}
	\end{algorithm}

	\begin{algorithm}[h]
		\caption{Linesearchprocedure($x, \alpha, d, \gamma, \gamma_1, \gamma_2$)}
		\label{alg:lsep}
		\begin{algorithmic}
			\par\vspace*{0.1cm}
			\IF{$f(R(x_k, \alpha d )) > f(x) - \gamma \alpha^2$} 
			\STATE $(0, \gamma_1 \alpha)$
			\ENDIF
			\WHILE{ $f(R(x_k, \alpha d )) < f(x) - \gamma \alpha^2 $}  
			\STATE Set $\alpha =  \gamma_2 \alpha$
			\ENDWHILE
			\STATE \textbf{Return} $(\alpha/\gamma_2, \alpha/\gamma_2)$			
		\end{algorithmic}
	\end{algorithm}
	
    We now proceed to prove the asymptotic convergence of this method.
	\begin{lemma} \label{l:splem}
		We have, at every iteration $k$, that the following inequality holds:
		\begin{equation} \label{eq:alphakgamma1}
			-\Sc{\grad f(x_k)}{p_k^{j(k)}} < \tilde{\alpha}_{k + 1}^{j(k)} \frac{\gamma_2}{\gamma_1} (2LB^2 + \gamma).
		\end{equation}
	\end{lemma}
	\begin{proof}
		It is immediate to check that we must always have
		\begin{equation}
			f(R(x_k,  \Delta_k  p_k^{j(k)})) > f(x_k) - \gamma \Delta_k^2,
		\end{equation}
		for $\Delta_k = \frac{1}{\gamma_1} \td{\alpha}_{k + 1}^{j(k)}$ if the Linesearchprocedure terminates at the second line, and $\Delta_k = \gamma_2\td{\alpha}_{k + 1}^{j(k)} $ if the Linesearchprocedure terminates in the last line. Then in both cases 
		\begin{equation}
			-\Sc{\grad f(x_k)}{p_k^{j(k)}} < \Delta_k (2LB^2 + \gamma) \leq \tilde{\alpha}_{k + 1}^{j(k)} \frac{\gamma_2}{\gamma1} (2LB^2 + \gamma) \, ,
		\end{equation}
		where we used Lemma \ref{l:deltaklem} in the first inequality.
	\end{proof}
	\begin{theorem}
		For $\{x_k\}$ generated by Algorithm \ref{alg:dse}, we have
		\begin{equation}
			\lim_{k \rightarrow \infty} \n{\grad f(x_k)} \rightarrow 0 \, .
		\end{equation}
	\end{theorem}
	\begin{proof}
		Let $\bar{\alpha}_k = \max_{j \in [1:K]} \tilde{\alpha}_{k + 1}^{j(k)}$, so that
		$	\bar{\alpha}_k \rightarrow 0$ since $\tilde{\alpha}_{k}^{j(k)} \rightarrow 0$, reasoning as in the proof of Theorem \ref{th:ds}.
		As a consequence of Lemma \ref{l:splem} we have 
		\begin{equation} \label{eq:jrel}
			-\Sc{\grad f(x_k)}{p_k^{j(k)}} < \bar{\alpha}_k c_1\, ,
		\end{equation}
		for the constant $c_1 = \frac{\gamma_2}{\gamma_1} (2LB^2 + \gamma)$ independent from $j(k)$. \\
		It remains to bound $\Sc{\grad f(x_k)}{p_k^i}$ for $i \neq j$. 
		To start with,  we have the following bound:
		\begin{equation} \label{eq:gbound}
			\begin{aligned}
				& -\Sc{\grad f(x_{k})}{p^i_k} \leq  -\Sc{\grad f(x_{k + h})}{p^i_{k + h}} + |\Sc{\grad f(x_{k + h})}{p^i_{k + h}} - \Sc{\grad f(x_{k})}{p^i_k}| \\ 
				\leq & c_1 \bar{\alpha}_{k + h} + |\Sc{\grad f(x_{k + h})}{p^i_{k + h}} - \Sc{\grad f(x_{k})}{p^i_k}|\, ,	
			\end{aligned}
		\end{equation}
		for $h \leq K$ such that $k + h = j(i)$, and where in the second inequality we used \eqref{eq:jrel} with $k + h$ instead of $k$. For the second summand appearing in the RHS of \eqref{eq:gbound}, we can write the following bound
		\begin{equation} \label{eq:dbound}
			\begin{aligned}
				& |\Sc{\grad f(x_{k + h})}{p^i_{k + h}} - \Sc{\grad f(x_k)}{p^i_k}| = |\Sc{\grad f(x_{k + h})}{p^i_{k + h}} - \Sc{\tP_k^{k + h} \grad f(x_k)}{\tP_k^{k + h}p^i_{k}}| \\ 
				\leq &  |\Sc{\grad f(x_{k + h}) - \tP_{k}^{k + h} \grad f(x_{k}) }{ p^i_{k + h}}| + |\Sc{\tP_{k}^{k + h} \grad f(x_{k}) }{p^i_{k + h} - \tP_k^{k + h} p^i_k}| \\ 
				+ & |\Sc{\grad f(x_{k + h}) - \tP_{k}^{k + h} \grad f(x_{k}) }{p^i_{k + h} - \tP_k^{k + h} p^i_k}|  \\
				\leq & L_f \dist(x_k, x_{k + h}) \n{p^i_{k + h}} + \LG \n{\grad f(x_k)} \dist(x_{k + h}, x_k) + L_f L_{\Gamma} \dist(x_k, x_{k + h})^2 \\
				\leq & (\Lf B + \LG M_f + L_fL_{\Gamma} \dist(x_{k + h}, x_k))\dist(x_{k + h}, x_k)\, ,	
			\end{aligned}
		\end{equation}
		where in the second inequality we used the Cauchy-Schwartz inequality together with the Assumptions on the Lipschitz property of the iterates \eqref{eq:lipf} and \eqref{eq:mb}, while in the third inequality we used conditions \eqref{eq:bounded} and \eqref{eq:mf}.  \\
		We can  now bound $\dist(x_k, x_{k + h})$ as follows
		\begin{equation} \label{eq:sumbound}
			\begin{aligned}
				& \dist(x_{k + h}, x_k) \leq \sum_{l = 0} ^ {h - 1} \dist(x_{k + l + 1}, x_{k + l}) \\
				= & \sum_{l= 0}^{h - 1} \dist(x_{k + l}, R(x_{k + l}, \bar{\alpha}_{k + l} p_{k + l}^{j(k + l)})) \leq \sum_{l= 0}^{h - 1} \Lr \bar{\alpha}_{k + l} \n{p_{k + l}^{j(k + l)}} \\
				\leq & B \Lr \sum_{l= 0}^{h - 1} \bar{\alpha}_{k + l}  	\leq hB\Lr \max_{l \in [0:h-1]} \bar{\alpha}_{k + l} \\ 
				\leq & KB\Lr \max_{l \in [0:K]} \bar{\alpha}_{k + l}  \, ,
			\end{aligned}
		\end{equation}
		where we used \eqref{eq:rbounded} in the second inequality, \eqref{eq:bounded} in the third one, and $h \leq K$ in the last one. \\
		Now let $\Delta_k = \max_{l \in [0:K]} \bar{\alpha}_{k + l} $, so that in particular $\Delta_k \rightarrow 0$.
		We apply \eqref{eq:sumbound} to the RHS of \eqref{eq:dbound} and obtain
		\begin{equation}
			\begin{aligned}
				& |\Sc{\grad f(x_{k + h})}{p^i_{k + h}} - \Sc{\grad f(x_k)}{p^i_k}| \leq (L_f B + \LG M_f + L_f \LG c_2 \Delta_k)c_2 \Delta_k  \rightarrow 0\, ,
			\end{aligned}
		\end{equation}
		for $k \rightarrow \infty$ and $c_2 = KB\Lr$.
		Finally, for every $i \in [1:K]$ 
		\begin{equation} \label{eq:czero}
			- \Sc{\grad f(x_{k})}{p^i_k} \leq c_1 \bar{\alpha}_{k + h} + (L_f B + \LG M_f + L_f  \LG c_2 \Delta_k)c_2 \Delta_k  \rightarrow 0 \, ,
		\end{equation}
		and the thesis follows after observing that, by \eqref{eq:pspanningt},
		\begin{equation}
			\n{\grad f(x_k)} \leq \frac{1}{\tau} \max_{i \in [1:K]} -\Sc{\grad f(x_{k})}{p^i_k} \rightarrow 0   \, ,
		\end{equation}
		where the convergence of the gradient norm to zero is a consequence of
		\eqref{eq:czero}.
	\end{proof}

	\section{Nonsmooth objectives}\label{s:nonsmooth}
	Now we proceed to present and study direct search methods in the context where $f$ is Lipschitz continuous and bounded from below, but not necessarily continuously differentiable. The algorithms we devise are built around the ideas given in~\cite{fasano2014linesearch}, where the authors consider direct search methods for nonsmooth objectives in Euclidean space.
	\subsection{Clarke stationarity for nonsmooth functions on Riemannian manifolds}
	In order to perform our analysis, we first need to define the Clarke directional derivative for a point $x \in M$. The standard approach is to write the function in coordinate charts and take the standard Clarke derivative in an Euclidean space (see, e.g.,~ \cite{hosseini2013nonsmooth} and \cite{hosseini2017riemannian}). Formally, given a chart $(\varphi, U)$ at $x \in M$ and $v \in T_xM$, we define 
	\begin{equation} \label{eq:clarke}
		f^{\circ}(x, v) = \tilde{f}(\varphi(x), d\varphi(x)v)'\, ,
	\end{equation}
	for $\tilde{f}(y) = f(\varphi^{- 1}(y))$. 
The following lemma shows the relationship between definition \eqref{eq:clarke} and a directional derivative like object defined with retractions.   
	\begin{lemma} \label{l:clarkeR}
		If $(y_k, q_k) \rightarrow (x, d)$ and $t_k \rightarrow 0$,
		\begin{equation}
			f^{\circ}(x, d) \geq \limsup_{k \rightarrow \infty} \frac{f(R(y_k, t_kq_k)) - f(y_k)}{t_k} \, .
		\end{equation}
	\end{lemma}
	The proof is rather technical and we defer it to the appendix.
	\subsection{Refining subsequences}

	We now adapt the definition of refining subsequence used in the analysis of direct search methods (see, e.g., \cite{audet2002analysis, fasano2014linesearch}) to the Riemannian setting. 
	Let $(x_k, d_k)$ be a sequence in $T\MM$.
	\begin{definition}
		We say that the subsequence $\{x_{i(k)}\}$ is refining if $x_{i(k)} \rightarrow x $, and if for every $d \in T_x\MM$ with $\n{d}_x = 1$ there is a further subsequence $\{j(i(k))\}$ such that 
		\begin{equation}
			\lim_{k \rightarrow \infty} d_{j(i(k))} = d \, .
		\end{equation}
	\end{definition}
We now give a sufficient condition for a sequence to be refining.
	\begin{proposition}
		If $x_{i(k)} \rightarrow x^*$, $\bar{d}_{i(k)}$ is dense in the unit sphere, and $d_{i(k)} = \PP_{k}(\bar{d}_{i(k)})/\n{\PP_k(\bar{d}_{i(k)})}_k$ for $\PP_k(\bar{d}_{i(k)}) \neq 0$ and $d_{i(k)} = 0$ otherwise, then it holds that the subsequence $\{x_{i(k)}\}$ is refining.
	\end{proposition}
	\begin{proof}
		Fix $d \in T_{x^*} \MM$, with $\n{d}_{x^*} = 1$, and let $\bar{d} = d/ \n{d}$. By density, we have that $\bar{d}_{j(i(k))} \rightarrow \bar{d}$ for a proper choice of the subsequence $\{j(i(k))\}$. Then 
		\begin{equation}
			\lim_{k \rightarrow \infty } d_{j(i(k))} = \lim_{k \rightarrow \infty } \frac{\PP_k(\bar{d}_{j(i(k))})}{\n{\PP_k(\bar{d}_{j(i(k))})}_k} 
			=  \frac{\PP_{x^*}(\bar{d})}{\n{\PP_{x^*}(\bar{d})}_{x^*}} = \frac{\bar{d}}{\n{\bar{d}}_{x^*}} = d \, ,
		\end{equation}
		where in the second equality we used the continuity of $\PP_x$ and of the norm $\n{\cdot}_x$, and in the third equality we used $\PP_{x^*}(\bar{d}) = \bar{d}$ since $\bar{d} \in T_{x^*} \MM$ by construction.
	\end{proof}
	\subsection{Direct search for nonsmooth objectives}
	We present here our Riemannian Direct Search method based on Dense Directions (RDS-DD) for nonsmooth objectives.  The  scheme is presented in detail as Algorithm \ref{alg:dse3}. The algorithm performs three simple steps at an iteration $k$.  First, a given search direction is suitably projected onto the current  tangent space.  Then a tentative point is generated by retracting the step $\alpha_k d_k$ from the  tangent space  to the manifold. Such a point is then eventually accepted as the new iterate if a sufficient decrease condition of the objective function is satisfied (and the stepsize is expanded), otherwise the iterate stays the same (and the stepsize is reduced).

	\begin{algorithm}[h]
		\caption{RDS-DD}
		\label{alg:dse3}
		\begin{algorithmic}
			\par\vspace*{0.1cm}
			\STATE{\textbf{Input:} $x_0\in \mathbb{R}^n$, $\alpha_0 > 0$,  $\gamma > 0, \gamma_1 \in (0, 1), \gamma_2 \geq 1$, $\{\bar{d}_k\}$ dense in $S(0, 1)$}
			\FOR{ $k=0, 1,...$}
			\STATE{Let $d_k = \PP_{k}(\bar{d}_k)/\n{\PP_{k}(\bar{d}_k)}_k$ if $\PP_{k}(\bar{d}_k) \neq 0$, $0$ otherwise}
			\IF{$f(R(x_k, \alpha_k d_k )) \leq f(x) - \gamma \alpha_k^2$}
			\STATE{$x_{k + 1} = R(x_k, \alpha_k d_k )$, $\alpha_{k + 1} = \gamma_2 \alpha_k$}
			\ELSE
			\STATE $x_{k + 1} = x_k$, $\alpha_{k + 1} = \gamma_1 \alpha_k$
			\ENDIF
			\ENDFOR
		\end{algorithmic}
	\end{algorithm}
	
	Thanks to the theoretical tools previously introduced, we can easily prove that a suitable subsequence of unsuccessful iterations of the  RDS-DD method converges to a Clarke stationary point.
	\begin{theorem} \label{th:dsns}
		Let $\{x_k\}$ be generated by Algorithm \ref{alg:dse3}. If $\{x_{i(k)}\}$ is refining, with $ x_{i(k)} \rightarrow x^* $, and $i(k)$ is an unsuccessful iteration for every $k \in \mathbb{N} \cup \{0\}$, then $x^*$ is Clarke stationary.
	\end{theorem}
	\begin{proof}
		Clearly as in the smooth case $\alpha_k \rightarrow 0$ and in particular $\alpha_{i(k)} \rightarrow 0$. Since by assumption $i(k)$ is an unsuccessful step, we have, for every $i(k)$
		\begin{equation} \label{eq:inequns}
			f(R(x_{i(k)}, \alpha_{i(k)} d_{i(k)})) - f(x_{i(k)}) > -\gamma \alpha_{i(k)}^2 \, .
		\end{equation}
		Let $\{j(i(k)) \}$ be such that $d_{j(i(k))} \rightarrow d$, and let $y_k = x_{j(i(k))} $, $q_k = d_{j(i(k))}$, $t_k = \alpha_{j(i(k))}$. We have
		\begin{equation}
			\limsup_{k \rightarrow \infty} \frac{f(R(y_k, t_kq_k)) - f(y_k)}{t_k} \geq \limsup_{k \rightarrow \infty} -\gamma \alpha_{i(k)} = 0\, ,
		\end{equation}
		thanks to \eqref{eq:inequns}, and by applying Lemma \ref{l:clarkeR} we get
		\begin{equation}
			f^{\circ}(x^*, d) \geq \limsup_{k \rightarrow \infty} \frac{f(R(y_k, t_kq_k)) - f(y_k)}{t_k} \geq 0 \, ,
		\end{equation}
		which implies the thesis since $d$ is arbitrary.
	\end{proof}
	
	\subsection{Direct search with extrapolation for nonsmooth objectives}
		We present here our Riemannian Direct Search method with Extrapolation based on Dense Directions (RDSE-DD) for nonsmooth objectives.  The detailed scheme is given in Algorithm \ref{alg:dse2}. As we can easily see, the algorithm performs just two simple steps at an iteration $k$.  First, a given search direction is suitably projected on the current  tangent space.  Then a linesearch  is performed using Algorithm \ref{alg:lsep} to hopefully obtain a new point that guarantees a sufficient decrease. 
	\begin{algorithm}[h]
		\caption{RDSE-DD}
		\label{alg:dse2}
		\begin{algorithmic}
			\par\vspace*{0.1cm}
			\STATE \textbf{Input:} $x_0\in \mathbb{R}^n$, $\alpha_0 > 0$,  $\gamma > 0, \gamma_1 \in (0, 1), \gamma_2 \geq 1$, $\{\bar{d}_k\}$ dense in $S(0, 1)$.
			\FOR{$k=0, 1,...$}
			\STATE Let $d_k = \PP_{k}(\bar{d}_k)/\n{\PP_{k}(\bar{d}_k)}_k$ if $\PP_{k}(\bar{d}_k) \neq 0$, $0$ otherwise. 
			\STATE Compute $\alpha_k, \tilde{\alpha}_{k + 1}$ with \textbf{Linesearchprocedure}($\tilde{\alpha}_k, x_k, d_k, \gamma, \gamma_1, \gamma_2$)  
		   \STATE Set $x_{k+1}= R(x_k, \alpha_k d_k)$	
		   \ENDFOR 
		\end{algorithmic}
	\end{algorithm}
	
	Once again, by exploiting the theoretical tools previously introduced, we can straightforwardly prove that a suitable subsequence of the  RDSE-DD iterations converges to a Clarke stationary point. It is interesting to notice that, thanks to the use of the linesearch strategy, we are not restricted to considering unsuccessful iterations this time. 
	
	\begin{theorem} \label{th:dsextns}
		Let $\{x_k\}$ be generated by Algorithm \ref{alg:dse2}.
		If $\{x_{i(k)}\}$ is refining, with $ x_{i(k)} \rightarrow x^* $, then $x^*$ is Clarke stationary.
	\end{theorem}
	\begin{proof}
		Let $\beta_k = \td{\alpha}_{k}/\gamma_2$ if the linesearch procedure exits before the loop, and $\beta_k = \gamma_1 \td{\alpha}_{k + 1}$ otherwise. Clearly $\beta_k \rightarrow 0$, and by definition of the linesearch procedure, for every $k$
		\begin{equation} \label{eq:inequns2}
			f(R(x_k, \beta_k d_k)) - f(x_k) > -\gamma \beta_k^2 \, .
		\end{equation}
		The rest of the proof is analogous to that of Theorem \ref{th:dsns}.
	\end{proof}
	
	\section{Numerical results}\label{s:num}
	We now report the results of some numerical experiments of the algorithms described in this paper on a set of simple but illustrative example problems. The comparison among the algorithms is carried out by using data and performance profiles \cite{more2009benchmarking}. Specifically, let $S$ be a set of algorithms and $P$ a set of problems. For each $s\in S$ and $p \in P$, let $t_{p,s}$ be the number of function evaluations required by algorithm $s$ on problem $p$ to satisfy the condition 
\begin{equation}\label{eq:stop}
f(x_k) \leq f_L + \tau(f(x_0) - f_L)\, , 
\end{equation}
where $0< \tau < 1$ and $f_L$ is the best objective function value achieved by any solver on problem $p$. Then, performance and data profiles of solver $s$ are the following functions
\begin{eqnarray*}
\rho_s(\alpha) & = & \frac{1}{|P|}\left|\left\{p\in P: \frac{t_{p,s}}{\min\{t_{p,s'}:s'\in S\}}\leq\alpha\right\}\right|,\\
d_s(\kappa) & = & \frac{1}{|P|}\left|\left\{p\in P: t_{p,s}\leq\kappa(n_p+1)\right\}\right|\, ,
\end{eqnarray*}
where $n_p$ is the dimension of problem $p$. \\
We used a budget of $100(n_p+1)$ function evaluations in all cases and two different precisions for the condition \eqref{eq:stop}, that is $\tau\in \{10^{-1},10^{-3}\}$. We consider randomly generated instances of well-known  optimization problems over manifolds from \cite{absil2009optimization,boumal2020introduction,hosseini2019nonsmooth}.    
A brief description of those problems as well as the details of our implementation can be found in the appendix (see Sections~\ref{s:Id}, \ref{sapp:numtests} and \ref{s:nsp}). The size of the ambient space for the instances  varies from 2 to 200. We would finally like to highlight that, in Section~\ref{s:anr}, we report further detailed numerical results, splitting the problems by ambient space dimension: between 2 and 15 for small instances, between 16 and 50 for medium instances, and between 51 and 200 for large instances. \\

\subsection{Smooth problems}
In Figure \ref{fig:overall}, we include the results related to 8 smooth instances of problem \eqref{eq:opt} from \cite{absil2009optimization,boumal2020introduction}, each with 15 different problem dimensions (from 2 to 200), for a total number of 60 tested instances.  We compared our methods, that is RDS-SB and RDSE-SB, with the zeroth order gradient descent (ZO-RGD, \cite[Algorithm 1]{li2020zeroth}). \\	 
The results clearly show that RDSE-SB performs better than RDS-SB and ZO-RGD both in efficiency and reliability for both levels of precision.	By taking a look at the detailed results in Section \ref{s:anr}, we can also see how the gap between RDSE-SB and the other two algorithms gets larger as the problem dimension grows.
	\begin{figure}[h]
		\centering
		\begin{subfigure}[b]{0.21\textwidth}
			\includegraphics[width=\linewidth]{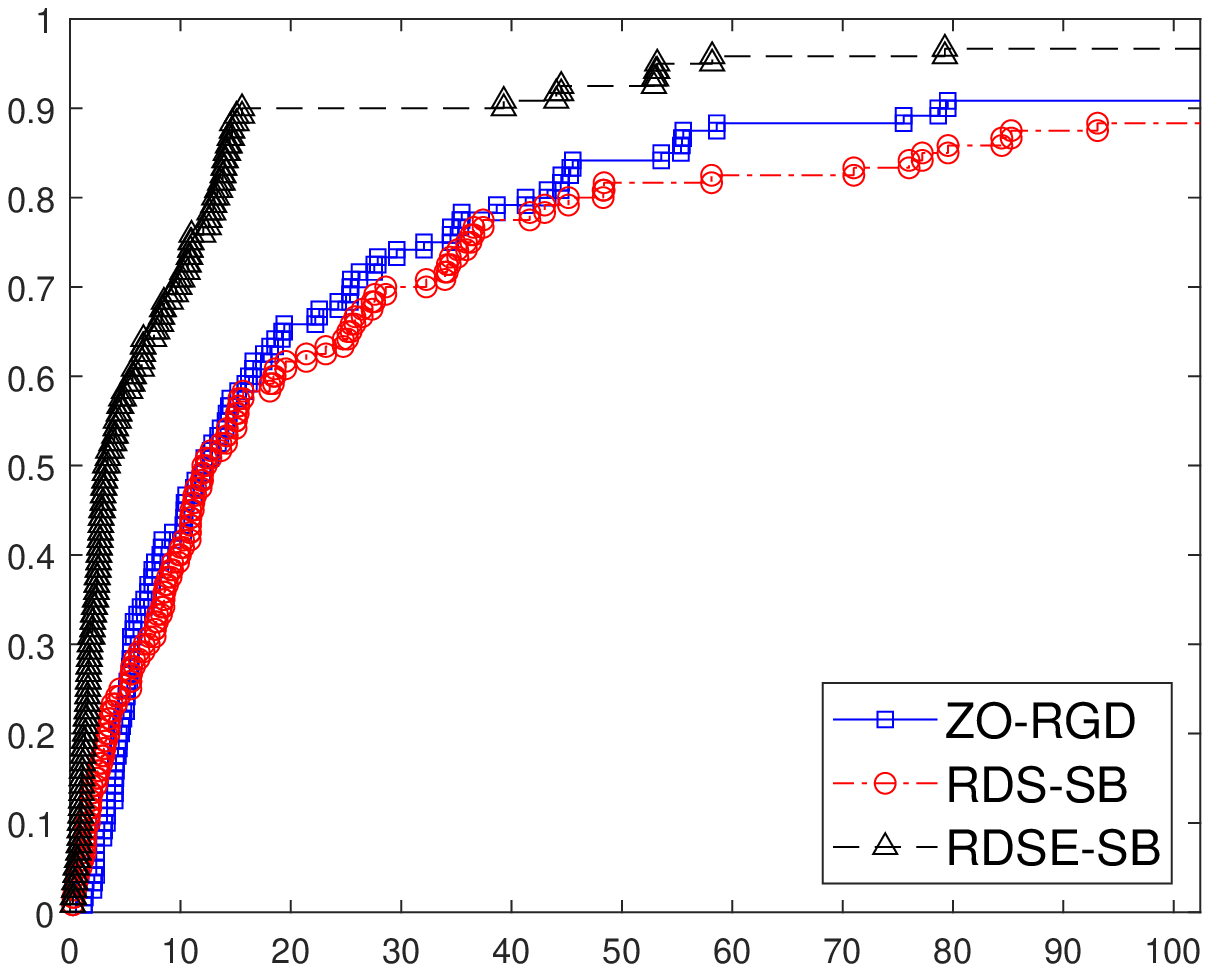}
			\caption{Data p., $\tau=10^{-1}$}
		\end{subfigure}	
		\begin{subfigure}[b]{0.21\textwidth}
			\includegraphics[width=\linewidth]{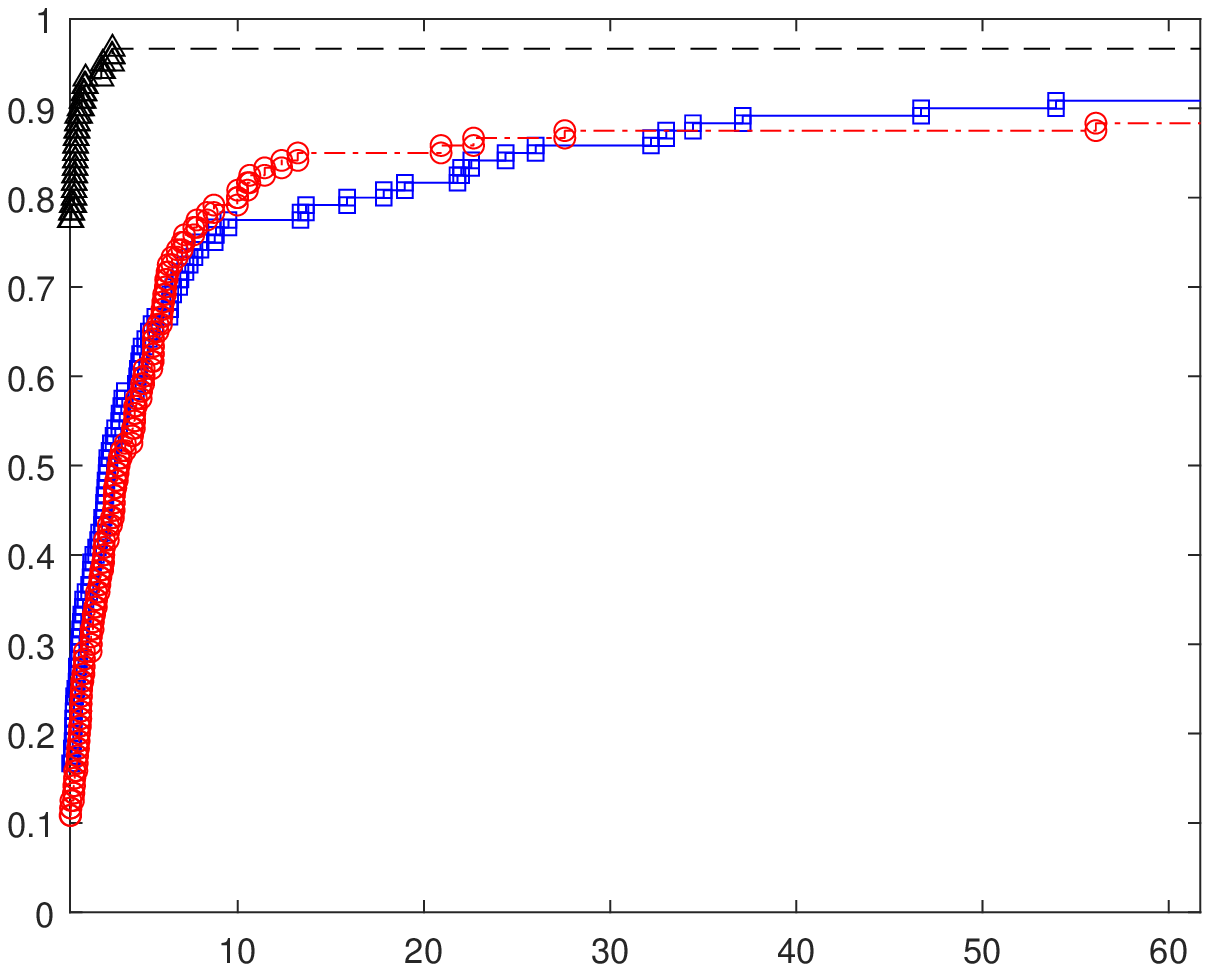}
			\caption{Perf. p., $\tau=10^{-1}$}
		\end{subfigure}
		\begin{subfigure}[b]{0.21\textwidth}
			\includegraphics[width=\linewidth]{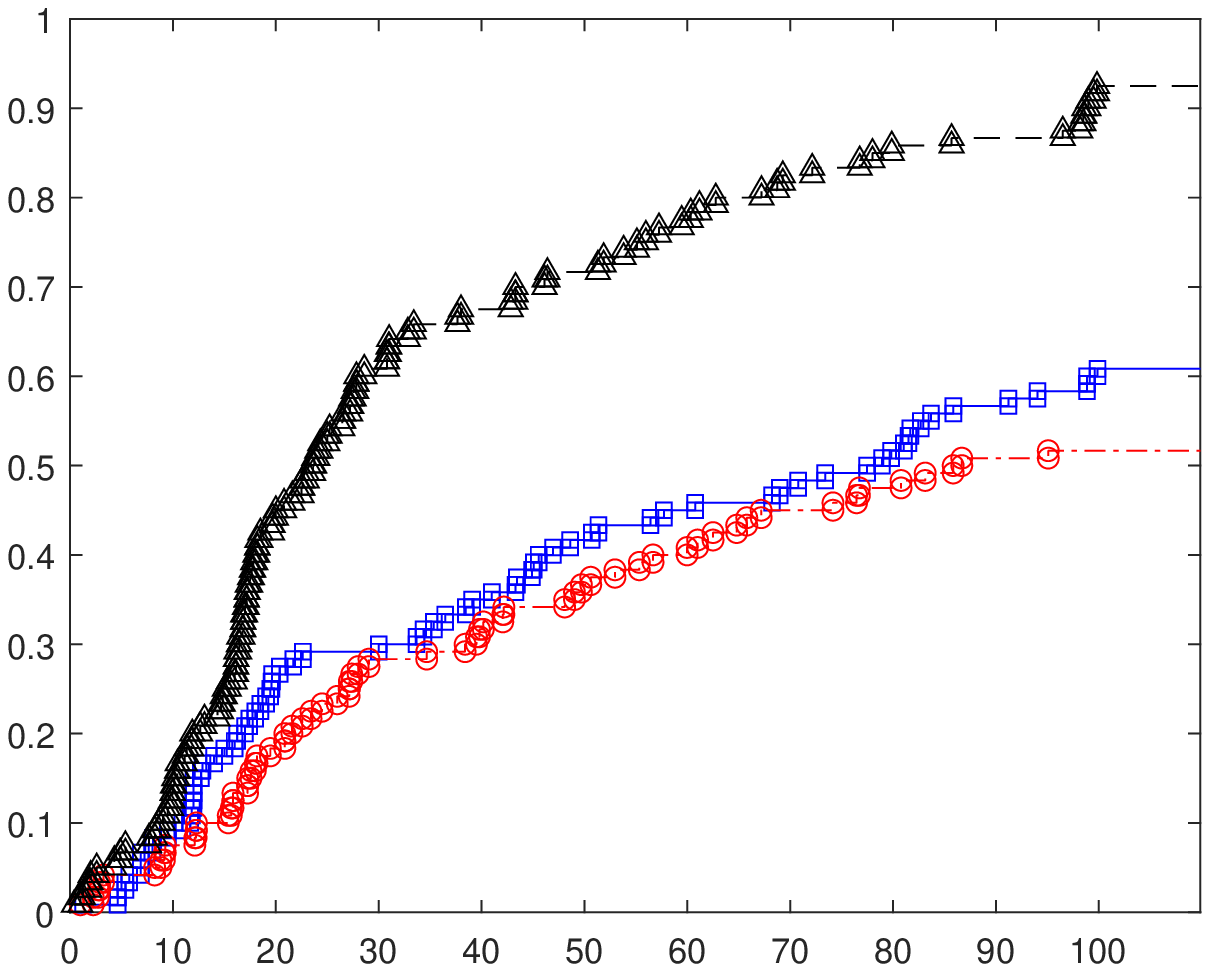}
			\caption{Data p., $\tau=10^{-3}$}
		\end{subfigure}
		\begin{subfigure}[b]{0.21\textwidth}
						\includegraphics[width=\linewidth]{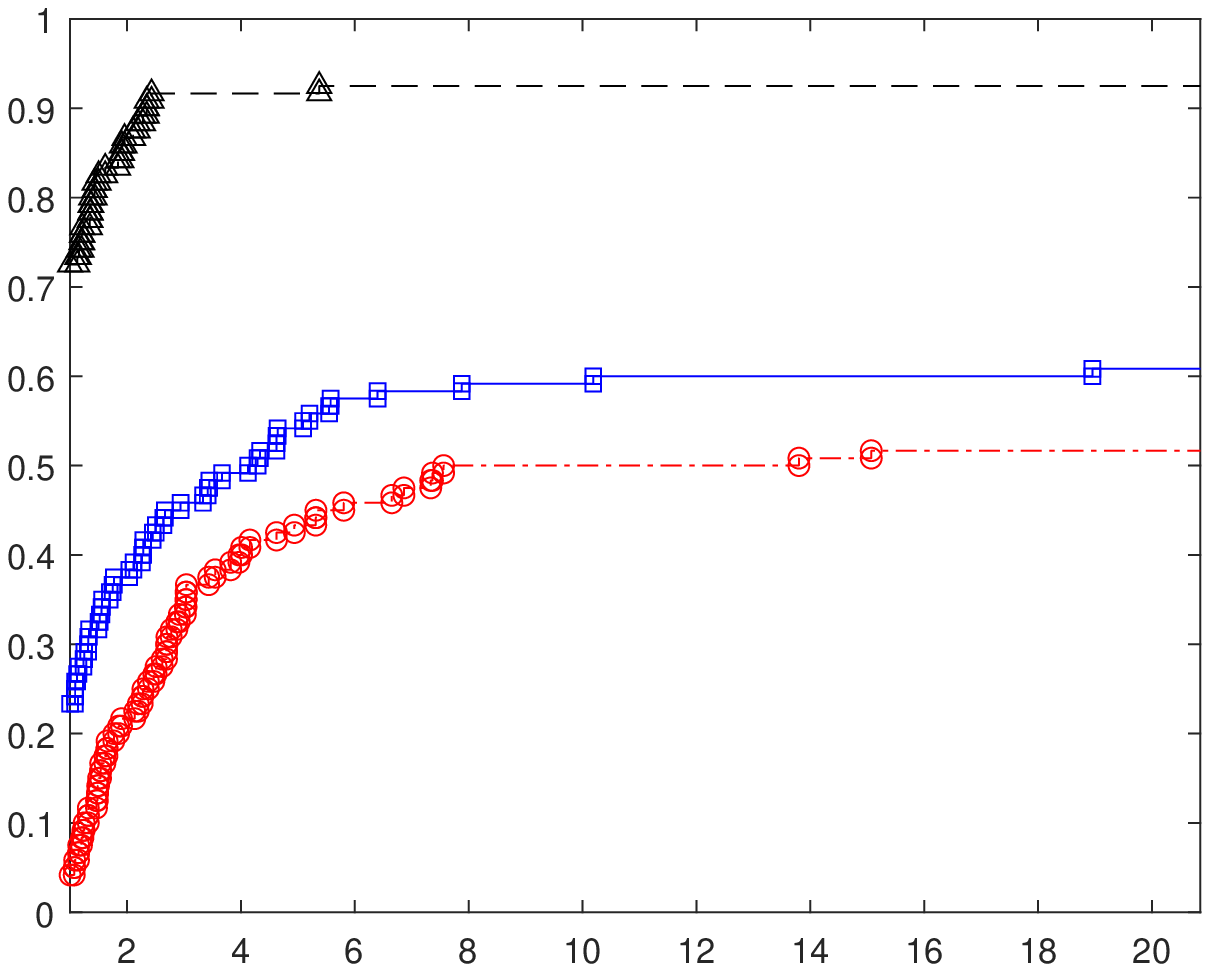}
			\caption{Perf. p., $\tau=10^{-3}$}
		\end{subfigure}
		\caption{Smooth case: results for all the instances}
		\label{fig:overall}
	\end{figure}

	\subsection{Nonsmooth problems}\label{sec:NSprob}
	We finally report a preliminary comparison between a direct search strategy and a linesearch strategy on two nonsmooth instances of \eqref{eq:opt} from \cite{hosseini2019nonsmooth}, each with 15 different problem sizes (from 2 to 200), thus getting a total number of 30 tested instances.

In the direct search strategy (RDS-DD+), we apply the RDS-SB method until $\alpha_{k + 1} \leq \alpha_{\epsilon}$, at which point we switch to the nonsmooth version RDS-DD. Analogously, in the linesearch strategy (RDSE-DD+), we apply the RDSE-SB method until $\max_{j \in [1:K]} \tilde{\alpha}_{k + 1}^j \leq \alpha_{\epsilon}$, at which point we switch to the nonsmooth version RDSE-DD. 	Both strategies use a threshold parameter $\alpha_{\epsilon} > 0$ to switch from the smooth to the nonsmooth DFO algorithm. We refer the reader to \cite{fasano2014linesearch} and references therein for other direct search strategies combining  coordinate and dense  directions.  \\ 
We report, in Figure \ref{fig:nonsmooth}, the comparison between the two considered strategies. As in the smooth case, the linesearch based strategy outperforms the simple direct search one. By taking a look at the detailed results in Section \ref{s:anr}, we can once again  see how the gap between the  algorithms gets larger as the problem dimension gets large enough.
	\begin{figure}[h]
	\centering
	\begin{subfigure}[b]{0.21\textwidth}
		\includegraphics[width=\linewidth]{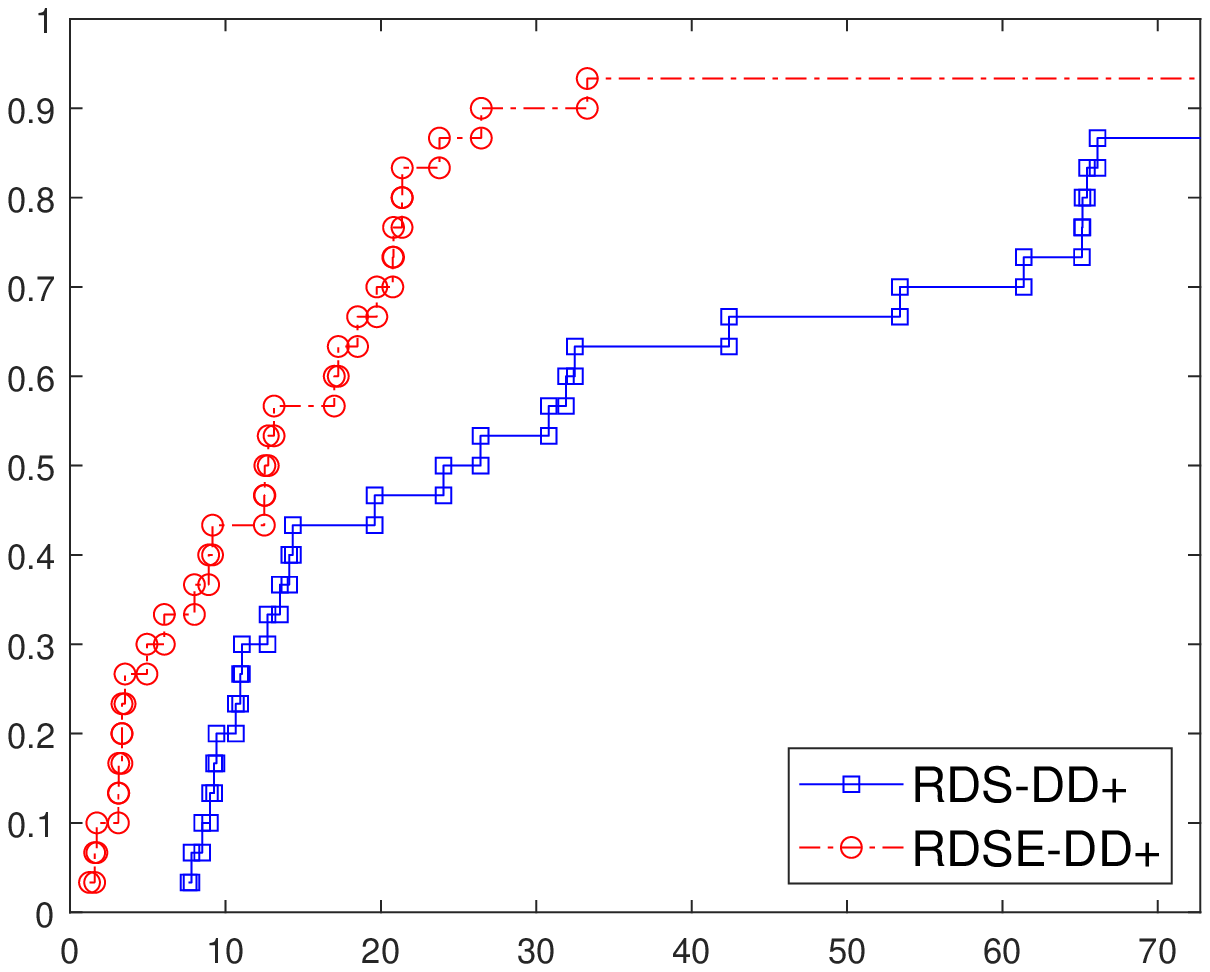}
		\caption{Data p., $\tau=10^{-1}$}
	\end{subfigure}	
	\begin{subfigure}[b]{0.21\textwidth}
		\includegraphics[width=\linewidth]{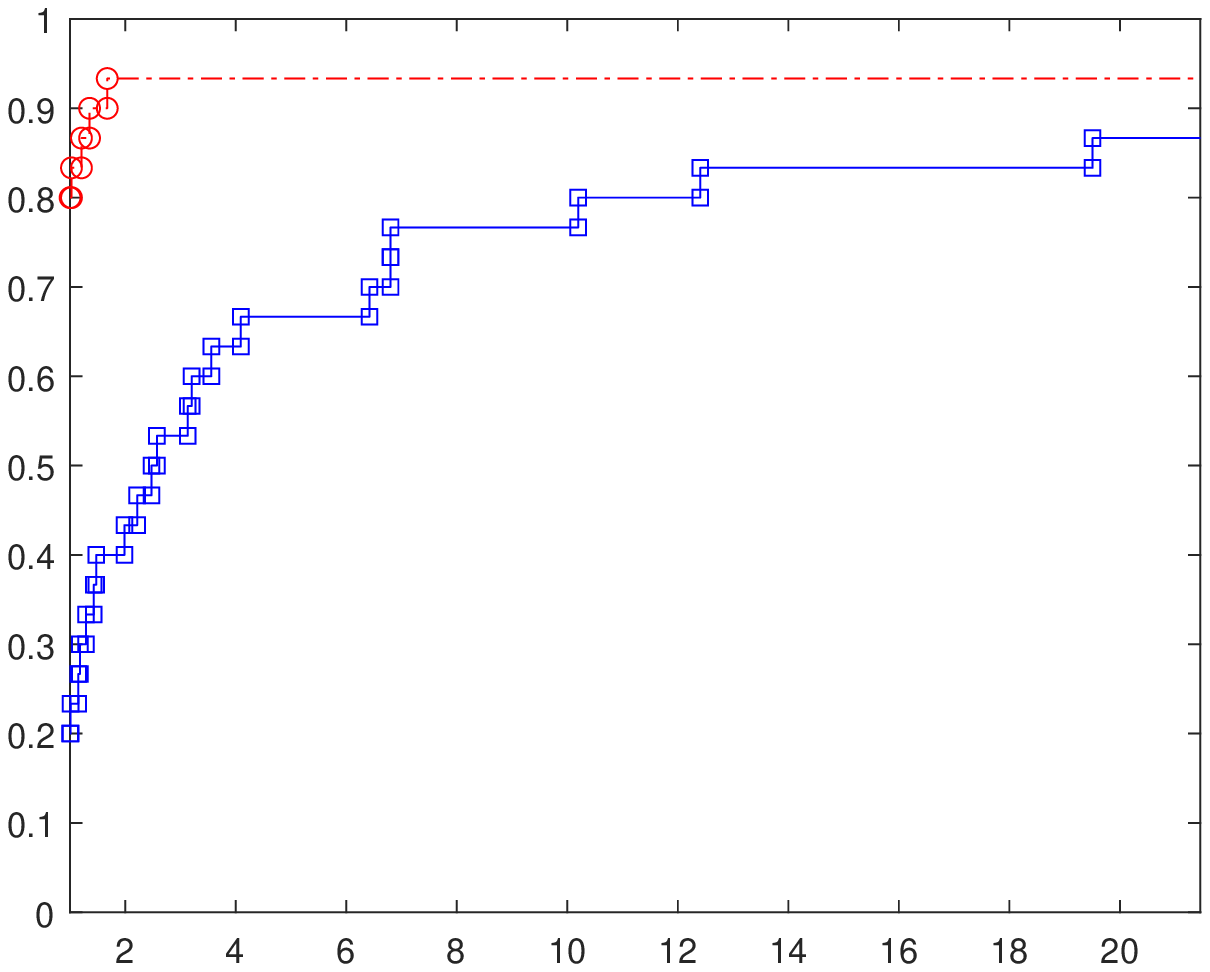}
		\caption{Perf. p., $\tau=10^{-1}$}
	\end{subfigure}
	\begin{subfigure}[b]{0.21\textwidth}
	\includegraphics[width=\linewidth]{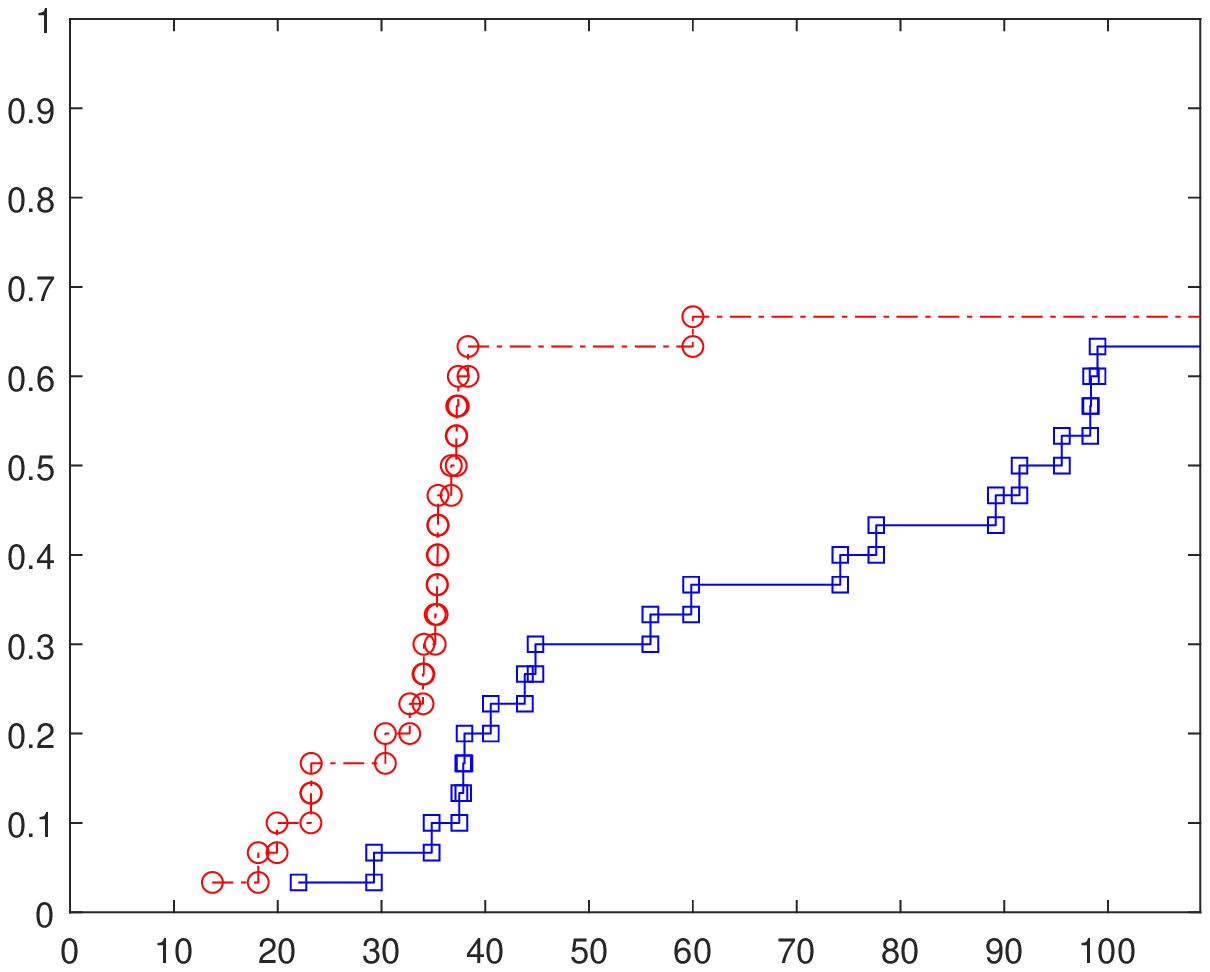}
	\caption{Data p., $\tau=10^{-3}$}
\end{subfigure}	
\begin{subfigure}[b]{0.21\textwidth}
	\includegraphics[width=\linewidth]{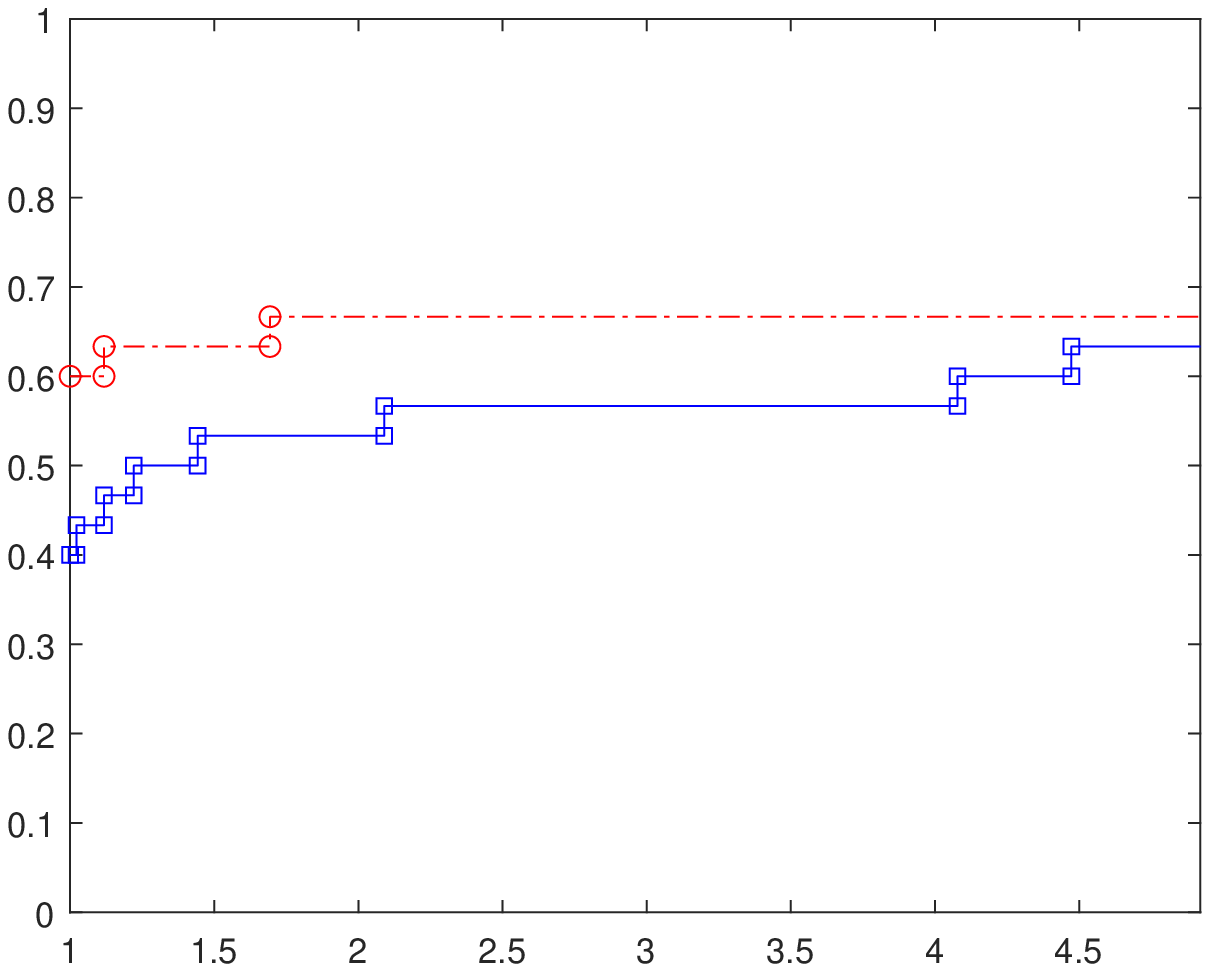}
	\caption{Perf. p., $\tau=10^{-3}$}
\end{subfigure}
	\caption{Nonsmooth case: results for all the instances}
\label{fig:nonsmooth}
\end{figure}	

\section{Conclusion}\label{s:con}
In this paper, we presented direct search algorithms with and without an extrapolation linesearch for minimizing functions over a Riemannian manifold. We found that, modulo modifications to account for the changing vector space structure with the iterations, direct search strategies provide guarantees of convergence for both smooth and nonsmooth objectives. We found also that in practice, in our numerical experiments, the extrapolation linesearch speeds up the performance of direct search in both cases, and it appears that it even outperforms a gradient approximation based zeroth order Riemannian algorithm in the smooth case. As a natural extension for future work, considering the stochastic case would be a reasonable next step.

	\bibliographystyle{siamplain}
	\bibliography{refs}

	\section{Appendix}

		\subsection{Proofs}
			In order to prove Proposition \ref{p:std} we first need the following lemma. 
	\begin{lemma} \label{l:hlem}
	For a Lipschitz continuous function $h: \R^m \rightarrow \R$, $\tilde{y}, \tilde{v} \in \R^m$, if $\tilde{y}_k \rightarrow \tilde{y}$, $\tilde{v}_k \rightarrow \td{v}$ and $t_k \rightarrow 0$ then
	\begin{equation}
		h^{\circ}(\td{y},\td{v}) \geq \limsup_{k \rightarrow \infty} \frac{h(\tilde{y}_k + t_k \tilde{v}_k) - h(\tilde{y}_k)}{t_k} \, .
	\end{equation}
\end{lemma}
\begin{proof}
	We have
	\begin{equation} \label{eq:otk}
		|h(\tilde{y}_k + t_k \tilde{v}_k) - h(\tilde{y}_k + t_k \td{v})| \leq t_k L_h\n{\td{v} - \td{v}_k} = o(t_k)\, ,
	\end{equation}
	with $L_h$ the Lipschitz constant of $h$. Then
	\begin{equation}
		\begin{aligned}
			& \limsup_{k \rightarrow \infty} \frac{h(\tilde{y}_k + t_k \tilde{v}_k) - h(\td{y}_k)}{t_k} = \limsup_{k   \rightarrow \infty} \frac{h(\tilde{y}_k + t_k \td{v}) + o(t_k) - h(\td{y}_k)}{t_k} \\
			= & \limsup_{k \rightarrow \infty} \frac{h(\tilde{y}_k + t_k \td{v}) - h(\td{y}_k)}{t_k} \leq h^{\circ}(\td{y}, \td{v}) \, ,	
		\end{aligned}
	\end{equation}
	where we used \eqref{eq:otk} in the first equality, and with the inequality true by definition of the Clarke derivative.
\end{proof}

		\begin{proof}[Proof of Proposition \ref{p:std}]
		Let $(\varphi)$ be a chart defined in a neighborhood $U$ of $x \in M$. We use the notation $(\td{x}, \td{d})= (\varphi(x), d\varphi(x)d)$ for $(x, d) \in T\MM$. We pushforward the manifold and the related structure with the chart $\varphi$, i.e. for $\bar{\varphi} = \varphi^{-1}$ we define $\tilde{f} = f\circ \bar{\varphi}$, $\td{U} = \varphi(U)$, $\tilde{R}(\tilde{y}, \tilde{d}) = R(y, d)$, for $d, q \in T_xM$ we define $g(\td{d}, \td{q}) = \Sc{d}{q}_x$, $\n{\td{d} - \td{q}}_{\td{x}} = \n{d - q}_x$,  and $\td{\Gamma}_{\td{x}}^{\td{y}}(\td{d}) = \Gamma_x^y(d)$. With slight abuse of notation we use $\dist(\td{x}, \td{y})$ to denote $\dist(x, y)$. We also define as $\grad \td{f}$ the gradient of $\td{f}$ with respect to the scalar product $g$, so that $g(\grad \td{f}(\td{x}), \td{d}) = \Sc{\nabla \td{f}(x)}{d}$ for any $\td{d} \in \R^m$. Importantly, by the equivalence of norms in $\R^m$ we can use $O(\n{\td{d}}_x)$ and $O(\n{\td{d}})$ interchangeably. \\
		We first prove \eqref{eq:taylor} in $x$ for some constant $L > 0$ and any $d$ with $\n{d} \leq B$ for some $B > 0$. Equivalently, we want to prove
		\begin{equation} \label{eq:th}
			\td{f}(\td{R}(\td{x}, \td{d})) \leq \td{f}(\td{x}) + g(\grad \td{f}(\td{x}), \td{d}) + \frac{L}{2}\n{\td{d}}_{\td{x}}^2 \, .
		\end{equation}
		for $\td{d}$ s.t. $\n{\td{d}} \leq B$. \\
By compactness we can choose $(\varphi, U)$ and $B > 0$ in such a way that, for every $\td{y} \in \td{U}_1 \subset \td{U}$ and $\td{d}$ with $\n{\td{d}}_{\td{y}} \leq B$ we have $\td{R}(\td{y}, \td{d}) \in \td{U}_2 \subset \td{U}$, with $\td{U}_2$ compact and $B > 0$ independent from $\td{x}, \td{y}, \td{d}$.  \\
First, since $\td{R}$ is in particular $C^1$ regular
\begin{equation}
	\tilde{R}(\tilde{x}, \tilde{d}) = \tilde{x} + O(\n{\td{d}}_{\td{x}})\, ,
\end{equation}
and by smoothness of the parallel transport
\begin{equation} \label{eq:gammasmooth}
	\td{\Gamma}_{\td{x}}^{\td{y}} \td{q} = \td{q} + O(\n{\td{x} - \td{y}}) \, .
\end{equation}
Furthermore,
\begin{equation}
	\grad \td{f}(\td{x} + \td{q}) =\td{\Gamma}_{\td{x}}^{\td{x} + \td{q}}  \grad \td{f}(\td{x}) + O(\dist(\td{x}, \td{x} + \td{q})) \, ,
\end{equation}
by the Lipschitz continuity assumption \eqref{eq:lipf}, and consequently
\begin{equation} \label{eq:gradR}
	\begin{aligned}
  & \grad \td{f}(\td{R}(\td{x}, \td{q})) =\td{\Gamma}_{\td{x}}^{\td{R}(\td{x}, \td{q})}  \grad \td{f}(\td{x}) + O(\dist(\td{x}, \td{R}(\td{x}, \td{q}))) \\
 = & \td{\Gamma}_{\td{x}}^{\td{R}(\td{x}, \td{q})}  \grad \td{f}(\td{x}) + O(\n{\td{q}}) \, ,		
	\end{aligned}
\end{equation} 
where we used \eqref{eq:rbounded} in the last equality. \\
Finally, since, $\frac{d}{dt}\td{R}(\td{x}, t\td{d})$ is $C^1$ regular, we also have
\begin{equation} \label{eq:ddt}
	\begin{aligned}
		 &\frac{d}{dt} \td{R}(\td{x}, t\td{q})|_{t = h} = \frac{d}{dt} \td{R}(\td{x}, t\td{q})|_{t = 0}  + 
	O(\n{h \td{q}}) \\ 
	= & \td{q} + O(h\n{\td{q}}) = \td{\Gamma}_{\td{x}}^{R(\td{x}, h\td{q})}\td{q} + O(\n{R(\td{x}, h\td{q}) - \td{x}}) + O(h \n{\td{q}}) = \td{\Gamma}_{\td{x}}^{R(\td{x}, h\td{q})}\td{q} + O(h \n{\td{q}})\, ,
	\end{aligned}
\end{equation}
where we used \eqref{eq:gammasmooth} in the third equality, and \eqref{eq:rbounded} in the last one. Again by compactness, for $\td{y} \in \td{U}_1$, $t \leq 1$, $\n{\td{q}}, \n{\td{d}} \leq B$ the implicit constants can be taken with no dependence from the variables. \\
Now for $\td{d}$ s.t. $\td{d} \leq B$ define $\td{q} = B \td{d}/\n{\td{d}}$, so that $\td{d} = \bar{t} \td{q}$ for $\bar{t} =\n{\td{d}}/B$. Then we obtain \eqref{eq:th} reasoning as follows:
\begin{equation}
	\begin{aligned}
		& \td{f}(\td{R}(\td{x}, \td{d})) - \td{f}(\td{R}(\td{x}, 0)) = \td{f}(\td{R}(\td{x}, \bar{t} q)) - \td{f}(\td{R}(\td{x}, 0)) \\ 
		= & \int_{0}^{\bar{t}} \frac{d}{dt} \td{f}(\td{R}(\td{x} + t\td{q})) dt = \int_{0}^{\bar{t}} g(\grad f(\td{R}(\td{x}, t\td{q})), \frac{d}{dt} \td{R}(\td{x}, t\td{d})) dt \\
= & \int_{0}^{\bar{t}} g(\td{\Gamma}_{\td{x}}^{\td{R}(\td{x}, t\td{q})}  \grad \td{f}(\td{x}) + O(t\n{\td{q}}),  \td{\Gamma}_{\td{x}}^{\td{R}(\td{x},t\td{d})}\td{d} + O(t \n{\td{q}})) dt  \\ 
= & \int_{0}^{\bar{t}} \left(g(\td{\Gamma}_{\td{x}}^{\td{R}(\td{x}, t\td{q})} \grad \td{f}(\td{x}), \td{\Gamma}_{\td{x}}^{\td{R}(\td{x},t\td{d})}\td{d}) +  O(t\n{\td{q}})\right) dt \\
= & g(\grad f(\td{x}), \td{d}) + O(\bar{t}^2 \n{\td{q}}) = g(\grad f(\td{x}), \td{d}) + O(\n{\td{d}}^2)	\, ,
	\end{aligned}			
\end{equation}
where we used \eqref{eq:gradR} and \eqref{eq:ddt} in the fourth inequality. To conclude, notice that the above argument does not depend from the choice of $\td{x} \in \td{U}_1$, so that it can be extended to every $\td{y} \in \td{U}_1$ and then by compactness to every $y \in M$.
\end{proof}

		\begin{proof}[Proof of Lemma \ref{l:clarkeR}]
		With the notation introduced in the proof of Proposition \ref{p:std}, without loss of generality we assume that $U$ is bounded and that $\varphi$ can be extended to a neighborhood containing the closure of $U$. \\
		First, since pushforward $\td{R}$ of a $C^2$ retraction on $\R$ is a $C^2$ retraction itself of $T \R^m$ on $\R^m$, we have the Taylor expansion
		\begin{equation} \label{eq:rtaylor}
			\td{R}(\td{y}, \td{v}) = \td{y} + \td{v} + O(\n{\td{v}}^2) \, ,
		\end{equation}
		with the implicit constant uniform for $\td{y}$ varying in $\td{U}$ and $\td{v}$ chosen in $\R^m$. \\
		Second, for any fixed constant $B> 0$, by continuity we have 
		\begin{equation} \label{eq:bgamma}
			\n{\td{\Gamma}_{\td{x}}^{\td{x}_k}\td{q} - \td{q}} \leq O\left( \n{\td{x} - \td{x}_k} \right)\, ,
		\end{equation}
		for $k \rightarrow \infty$, $\td{q} \in \R^m$ with $\n{\td{q}} \leq B$, and with a uniform implicit constant. \\
		Therefore
		\begin{equation} \label{eq:dktdk}
			\begin{aligned}
				& \n{\td{d}_k - \td{d}} \leq \n{\td{d}_k - \td{\Gamma}_{\td{x}}^{\td{x}_k}\td{d}} + \n{\td{\Gamma}_{\td{x}}^{\td{x}_k}\td{d} - \td{d}} \leq O\left( \n{\td{d}_k - \td{\Gamma}_{\td{x}}^{\td{x}_k}(\td{d})}_{\td{x}} \right) + O\left(\n{\td{x} - \td{x}_k}\right) 	\\
				& = O\left(\n{d_k - \Gamma_x^{x_k}(d)}_{x} \right) +  O\left( \n{\td{x} - \td{x}_k} \right)  = o(1) \, ,
			\end{aligned}
		\end{equation} 
		where in the second inequality we used \eqref{eq:bgamma}, and in the last equality we used $d_k \rightarrow d$ together with $\td{x}_k \rightarrow \td{x}$. \\
		Let now $\td{v}_k = (\td{R}(\td{x}_k, t_k\td{d}_k) - \td{x}_k)/t_k$. Then
		\begin{equation} \label{eq:vkd}
			\begin{aligned}
				& \n{\td{v}_k - \td{d}} = \frac{1}{t_k}\n{\td{R}(\td{x}_k, t_k\td{d}_k) - \td{x}_k - t_k\td{d}} \leq \frac{1}{t_k} 
				(\n{R(\td{x}_k, t_k\td{d}_k) - \td{x}_k - t_k\td{d}_k} + t_k\n{d_k - \td{d}_k}) \\ 
				= & \frac{1}{t_k} (O(t_k^2 \n{\tilde{d}_k}^2) + t_k o(1)) = o(1)\,  ,	
			\end{aligned}
		\end{equation}
		where we used \eqref{eq:rtaylor} and \eqref{eq:dktdk} for the first and the second summand in the second equality. In other words, $\td{v}_k \rightarrow \td{d}$. To conclude, 
		\begin{equation}
			\begin{aligned}
				& \limsup_{k \rightarrow \infty} \frac{f(R(y_k, t_kd_k)) - f(y_k)}{t_k} = \limsup_{k \rightarrow \infty}    \frac{\td{f}(\td{R}(\td{y}_k, t_k\td{d}_k)) - \td{f}(\td{y}_k)}{t_k} \\ 
				= & \limsup_{k \rightarrow \infty} \frac{\td{f}(\td{y}_k + t_k \td{v}_k) - \td{f}(\td{y}_k)}{t_k} \geq \td{f}^{\circ}(\td{x}, \td{d}) = f^{\circ}(x, d) \, ,
			\end{aligned}
		\end{equation}
		where in the inequality we were able to apply \eqref{l:hlem} because $\td{v}_k \rightarrow \td{d}$ by \eqref{eq:vkd}.
	\end{proof}

\subsection{Implementation details} \label{s:Id}
For all the problems, the manifold structure we used was the one available in the MANOPT library \cite{manopt}. 
 After a basic tuning phase, we set the algorithm parameters as follows:
we used $\gamma_1 = 0.61$, $\gamma_2 = 1$ and $\gamma= 0.77$ for Algorithm \ref{alg:ds},  $\gamma_1= 0.81$, $\gamma_2 = 3.12$ and $\gamma= 0.11$ for Algorithm \ref{alg:dse}, and the stepsize $1.64/n$ (recall that $n$ is the dimension of the ambient space) for the ZO-RGD method. \\ 
For the nonsmooth strategies RDS-DD+ and RDSE-DD+, we considered the same parameters of the smooth case for RDS-SB and RDSE-SB, setting $\alpha_{\epsilon} = 10^{-3}$, and for both RDS-DD and RDSE-DD used $\gamma_1= 0.95$, $\gamma_2= 2$, and $\gamma= 1$. \\
The positive spanning basis was obtained both in Algorithm \ref{alg:ds} and Algorithm \ref{alg:dse} by projecting the positive spanning basis $(e_1, ..., e_n, - e_1, ..., - e_n)$ of the ambient space $\mathbb{R}^n$ on the tangent space. The initial stepsize was set to $1$ for all the direct search methods, with no fine tuning. \\
We generated the starting point and the parameters related to the instances either with MATLAB rand function or by using the random element generators implemented in the MANOPT library. 
\subsection{Smooth problems}	\label{sapp:numtests}
We describe here the 8 smooth instances of problem \eqref{eq:opt} from \cite{absil2009optimization,boumal2020introduction}.

\subsubsection{Largest eigenvalue, singular value, and top singular values problem}
In the largest eigenvalue problem \cite[Section 2.3]{boumal2020introduction}, given a symmetric matrix $A \in S(n,n) = \{A \in \R^{n \times n} \ | \ A = A^\top \}$, we are interested in computing
\begin{equation} \label{p:leig}
	\max_{x \in \mathbb{S}^{n - 1}} x^\top Ax \, .
\end{equation}
The largest singular value problem \cite[Section 2.3]{boumal2020introduction} can be formulated generalizing \eqref{p:leig}: given $A \in \R^{m \times h}$, we are interested in 
\begin{equation}\label{p:lsv}
	\max_{x \in \Sb^{m - 1}, y \in \Sb^{h - 1}} x^\top Ay \, .
\end{equation}
Notice how the domain in \eqref{p:leig} and \eqref{p:lsv} are a sphere and the product of two spheres respectively. \\
Finally, to compute the sum of the top $r$ singular values, as explained in \cite[Section 2.5]{boumal2020introduction} it suffices to solve
\begin{equation}
	\max_{X \in S(m, r), Y \in S(h, r)} X^\top AY \, ,
\end{equation}
for $S(a,b)$ the Stiefel manifold with dimensions $(a, b)$.
\subsubsection{Dictionary learning}
The dictionary learning problem \cite[Section 2.4]{boumal2020introduction} can be formulated as
\begin{equation}\label{p:mainOT}
	\begin{array}{ll}
		\min & \n{Y -DC} + \lambda \n{C}_1, \\
		\quad \tx{s.t.} & D \in \R^{d \times h}, C\in \R^{h \times k}, \ \n{D_1} = ... = \n{D_h} = 1 \, , \\
	\end{array}
\end{equation} 
for a fixed $Y \in \R^{d \times k}$, $\lambda > 0$, $\n{\cdot}_1$ the $\ell_1 -$ norm, and $D_1, ..., D_h$ the columns of $D$. \\
In our implementation we smooth the objective by using a smoothed version $\n{\cdot }_{1, \varepsilon}$ of $\n{\cdot}_1$
\begin{equation}
	\n{C}_{1, \varepsilon} = \sum_{i, j} \sqrt{C_{i, j}^2 + \varepsilon^2} \, .
\end{equation}
In our tests, we generated the solution $\bar{C}$ using MATLAB sprand function, with a density of $0.3$, set the regularization parameter $\lambda$ to $0.01$ and $\varepsilon$ to $0.001$.
\subsubsection{Synchronization of rotations}
Let $\tx{SO}(d)$ be the special orthogonal group:
\begin{equation}
	\tx{SO}(d) = \{R \in \R^{d \times d} \ | \ R^\top R=I_d \tx{ and } \det(R) = 1\} \, .
\end{equation}
In the synchronization of rotations problem \cite[Section 2.6]{boumal2020introduction}, we need to find rotations $R_1, ..., R_h \in \tx{SO}(d)$ from noisy measurements $H_{ij}$ of $R_iR_j^{-1}$, for every $(i, j) \in E$, a subset of ${h \choose 2}$ (the set of couples of distinct elements in $[1:h]$). The objective is then
\begin{equation}
	\min_{\hat{R}_1, ..., \hat{R}_h \in \tx{SO}(d)} \sum_{(i,j) \in E} \n{\hat{R}_i - H_{ij} \hat{R}_j}^2 \, .
\end{equation}
In our tests, we considered the case $h = 2$ for simplicity.
\subsubsection{Low-rank matrix completion}
The low rank matrix completion problem \cite[Section 2.7]{boumal2020introduction} can be written, for a fixed matrix $M \in \R^{m \times h}$, as
\begin{equation}\label{p:lrm}
	\begin{array}{ll}
		\min & \sum_{(i, j) \in \Omega} (X_{ij} - M_{ij})^2, \\
		\quad s.t. &X \in \R^{m \times h}, \, \tx{rank}(X) = r \, ,
	\end{array}
\end{equation}
given a positive integer $r > 0$ and a subset of indices $\Omega \subset [1:m] \times [1:h]$. It can be proven that the optimization domain, that is the matrices in $\R^{m \times n}$ with fixed rank $r$, can be given a Riemannian manifold structure (see, e.g., \cite{vandereycken2010riemannian}).
\subsubsection{Gaussian mixture models}
In the Gaussian mixture model problem \cite[Section 2.8]{boumal2020introduction}, we are interested in computing a maximum likelihood estimation for a given set of observations $x_1, ..., x_h$:
\begin{equation}
	\max_{\substack{\hat{u}_1,..., \hat{u}_k \in \R^d \\ \hat{\Sigma}_1, ..., \hat{\Sigma}_k \in \tx{Sym}(d)^+, \\ w \in \Delta^{K - 1}_+}} \sum_{i = 1}^h \log\left( \sum_{k = 1}^K w_k \frac{1}{\sqrt{2\pi \det(\Sigma_k)}} e^{\frac{(x-\mu_k)^\top  \Sigma_k^{-1} (x-\mu_k)}{2}} \right) \, ,
\end{equation}
where $\tx{Sym}(d)^+$ is the manifold of positive definite matrices
\begin{equation}
	\tx{Sym}(d)^+ = \{X \in \R^{d \times d} \ | \ X = X^\top , \, X \succ 0 \}
\end{equation}
and $\Delta^{K - 1}_+$ is the subset of strictly positive elements of the simplex $\Delta^{K - 1}$, which can be given a manifold structure.
In our tests, we considered the case $K = 2$ and the reformulation proposed in \cite{hosseini2015matrix}, which does not use the unconstrained variables $(\hat{u}_1,..., \hat{u}_k)$. 
\subsubsection{Procrustes problem} \label{p:proc}
The Procrustes problem \cite{absil2009optimization} is the following linear regression problem, for fixed $A \in \R^{l \times n}$ and $B \in \R^{l \times p}$:
\begin{equation}
	\min_{x \in \mathcal M} \n{AX - B}_F^2 \, ,
\end{equation}
In our tests, we assumed the variable $X \in \R^{n \times p}$ to be in the Stiefel manifold $\tx{St}(n, p)$, a choice leading to the so called unbalanced orthogonal Procrustes problem.

\subsection{Nonsmooth problems} \label{s:nsp}
We report two nonsmooth problems taken from \cite{hosseini2019nonsmooth}.
\subsubsection{Sparsest vector in a subspace}
Given an orthonormal matrix $Q \in \R^{m \times n}$, the problem of finding the sparsest vector in the subspace generated by the columns of $Q$ can be relaxed as
\begin{equation} \label{p:nsp}
	\min_{x \in \mathbb{S}^{n - 1}} \n{Qx}_1 \, .
\end{equation}
\subsubsection{Nonsmooth low-rank matrix completion}
In the nonsmooth version of the low rank matrix completion problem \eqref{p:lrm} the Euclidean norm is replaced with the $l_1$ norm, so that in the objective we have a sum of absolute values:
\begin{equation}\label{p:lrmns}
	\begin{array}{ll}
		\min &\sum_{(i, j) \in \Omega} |X_{ij} - M_{ij}|, \\
		\quad s.t. &X \in \R^{m \times n}, \,  \tx{rank}(X) = r \, .
	\end{array}
\end{equation}
\subsection{Additional numerical results} \label{s:anr}
We include here the performance and data profiles split by problem size.
	\begin{figure}[h]
	\centering
		\begin{subfigure}[b]{0.21\textwidth}
	\includegraphics[width=\linewidth]{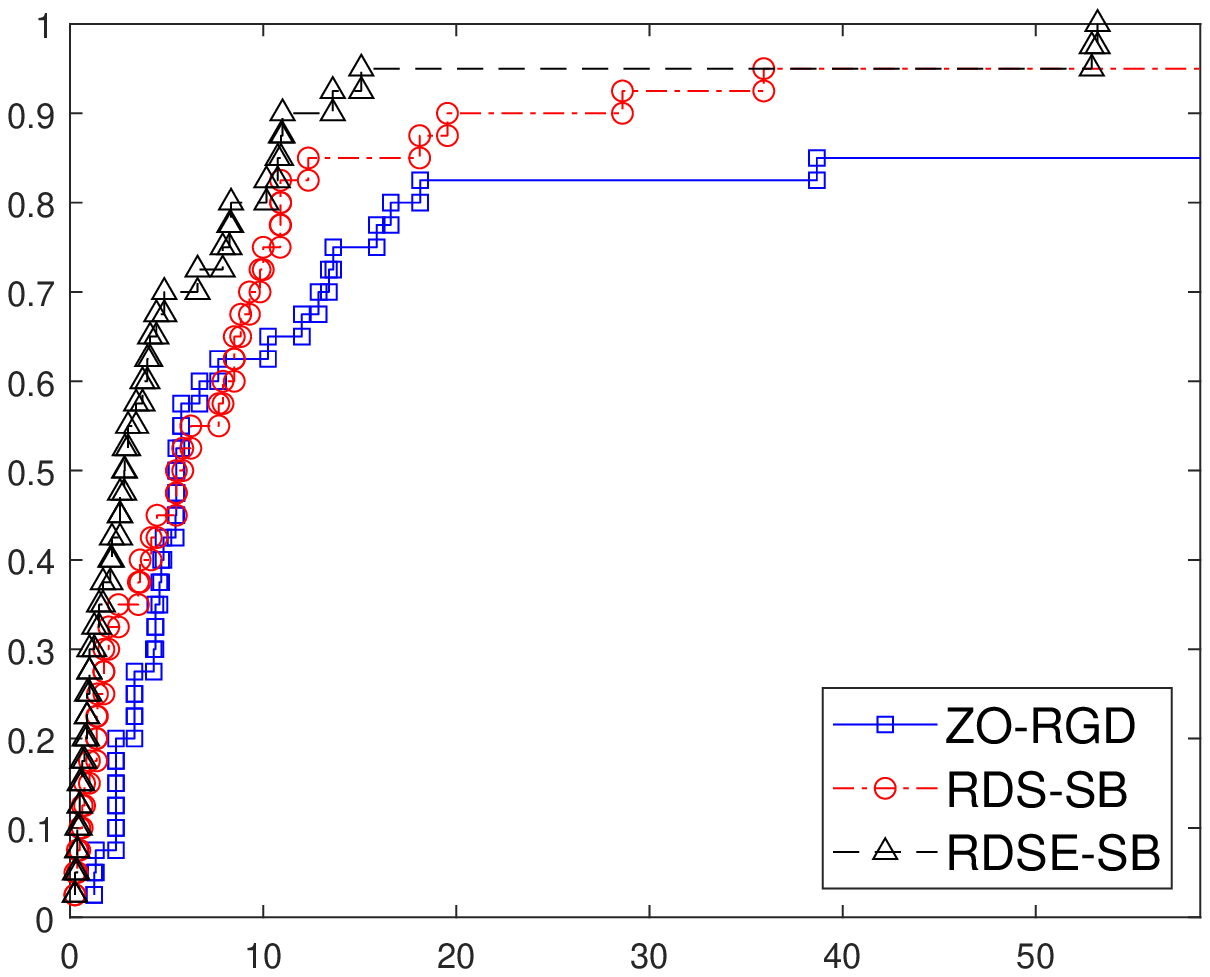}
	\caption{Data p., $\tau=10^{-1}$}
\end{subfigure}	
	\begin{subfigure}[b]{0.21\textwidth}
		\includegraphics[width=\linewidth]{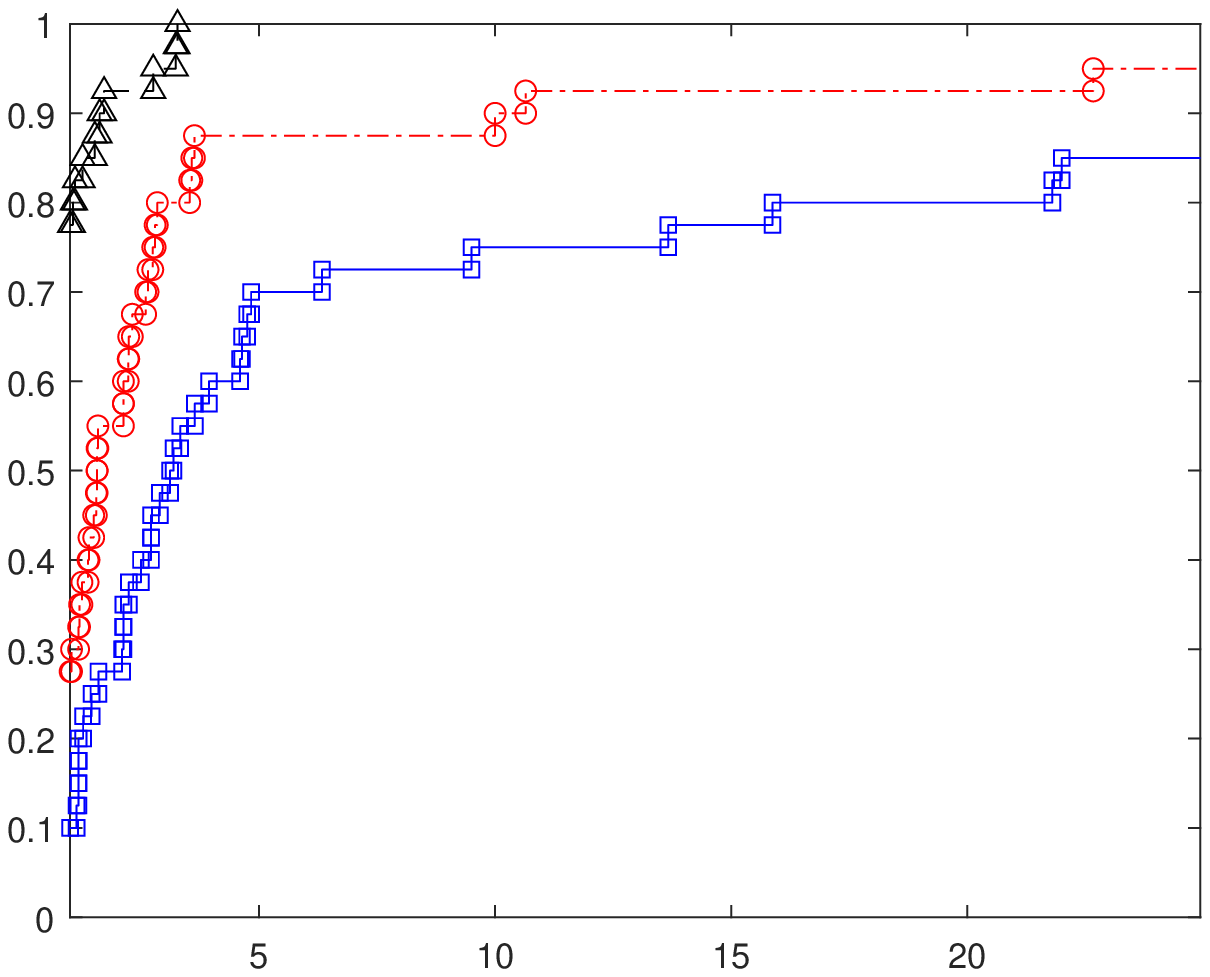}
		\caption{Perf. p., $\tau=10^{-1}$}
	\end{subfigure}	
	\begin{subfigure}[b]{0.21\textwidth}
		\includegraphics[width=\linewidth]{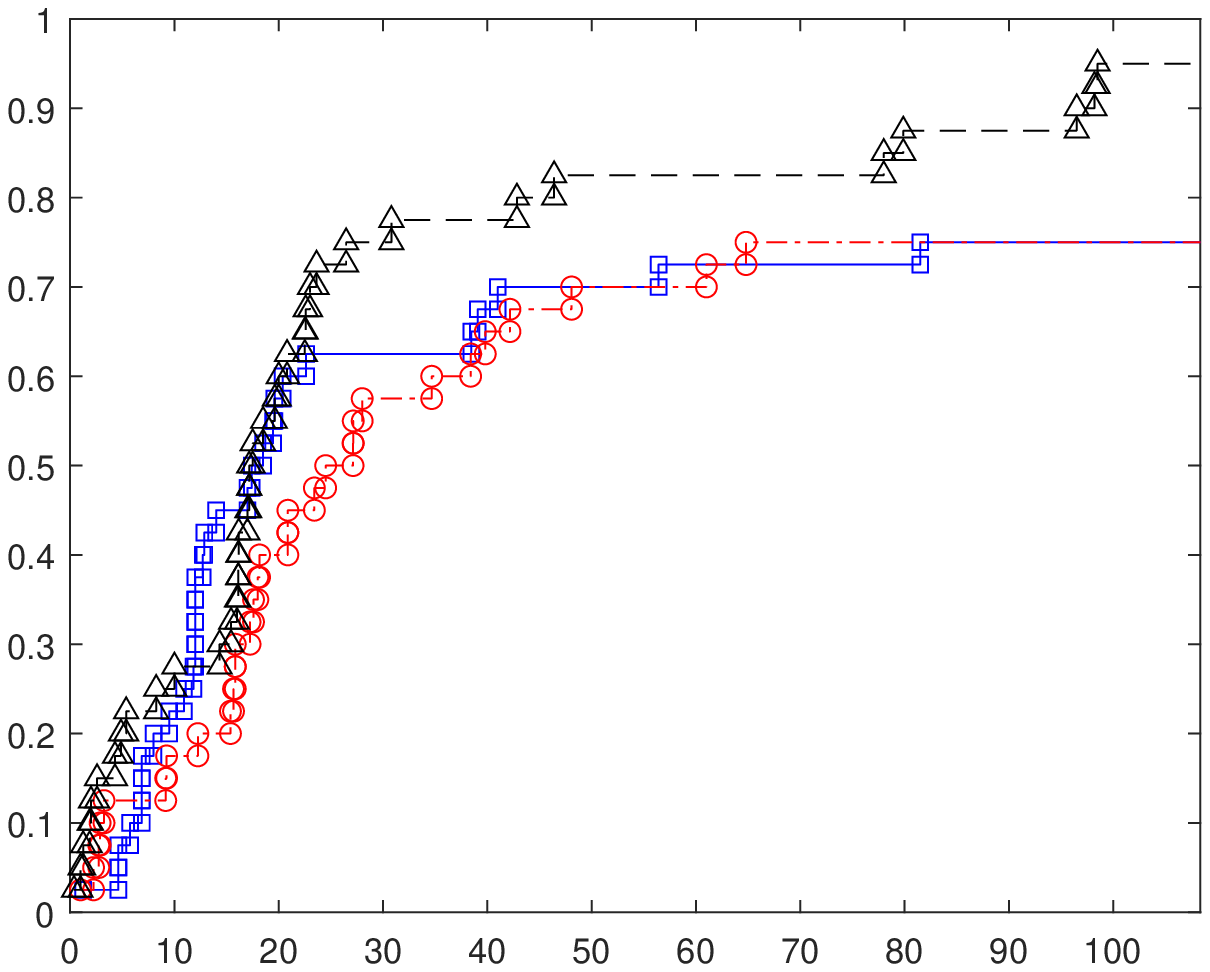}
		\caption{Data p., $\tau=10^{-3}$}
	\end{subfigure}	
	\begin{subfigure}[b]{0.21\textwidth}
		\includegraphics[width=\linewidth]{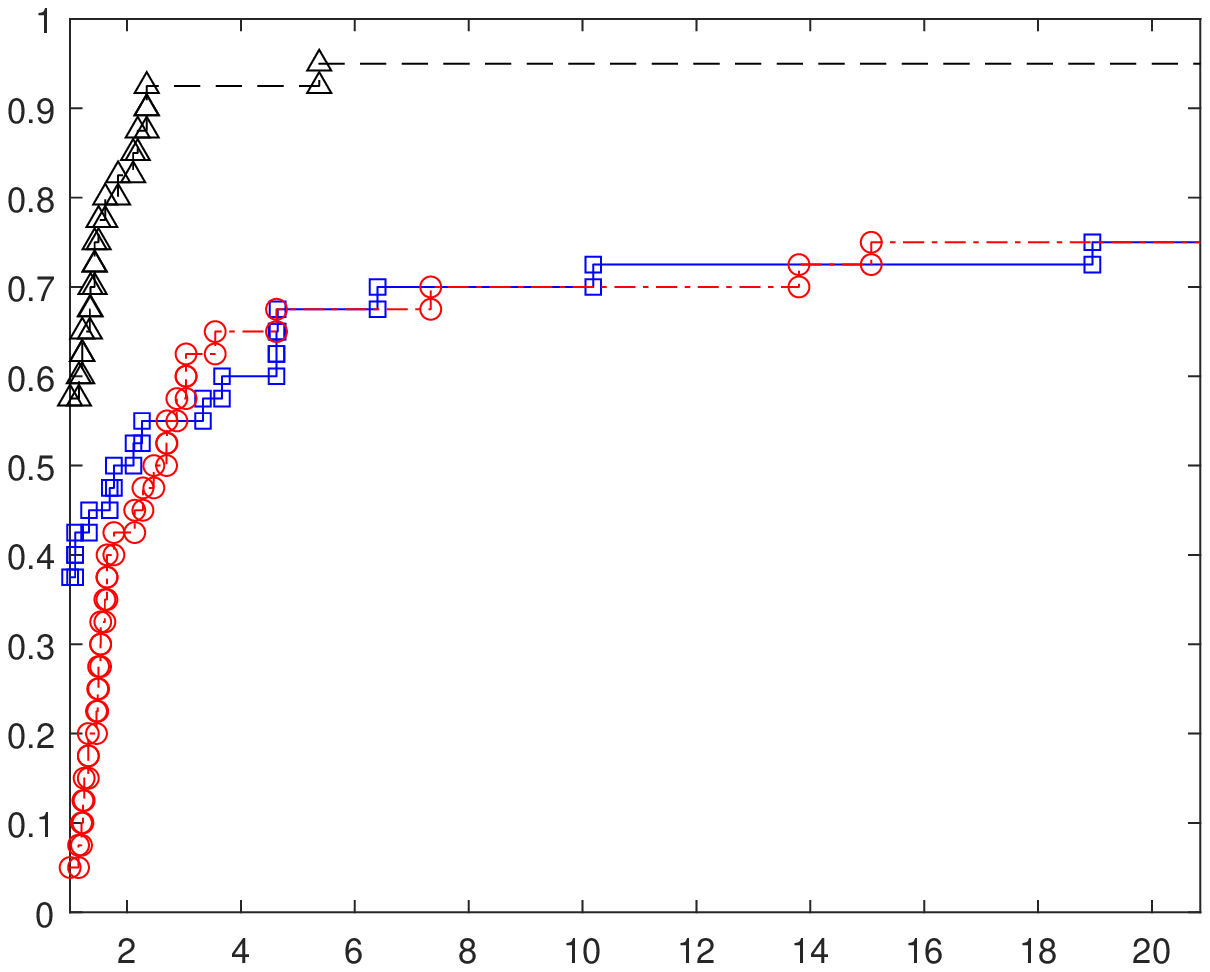}
		\caption{Perf. p., $\tau=10^{-3}$}
	\end{subfigure}	
\vspace{3mm}

		\begin{subfigure}[b]{0.21\textwidth}
	\includegraphics[width=\linewidth]{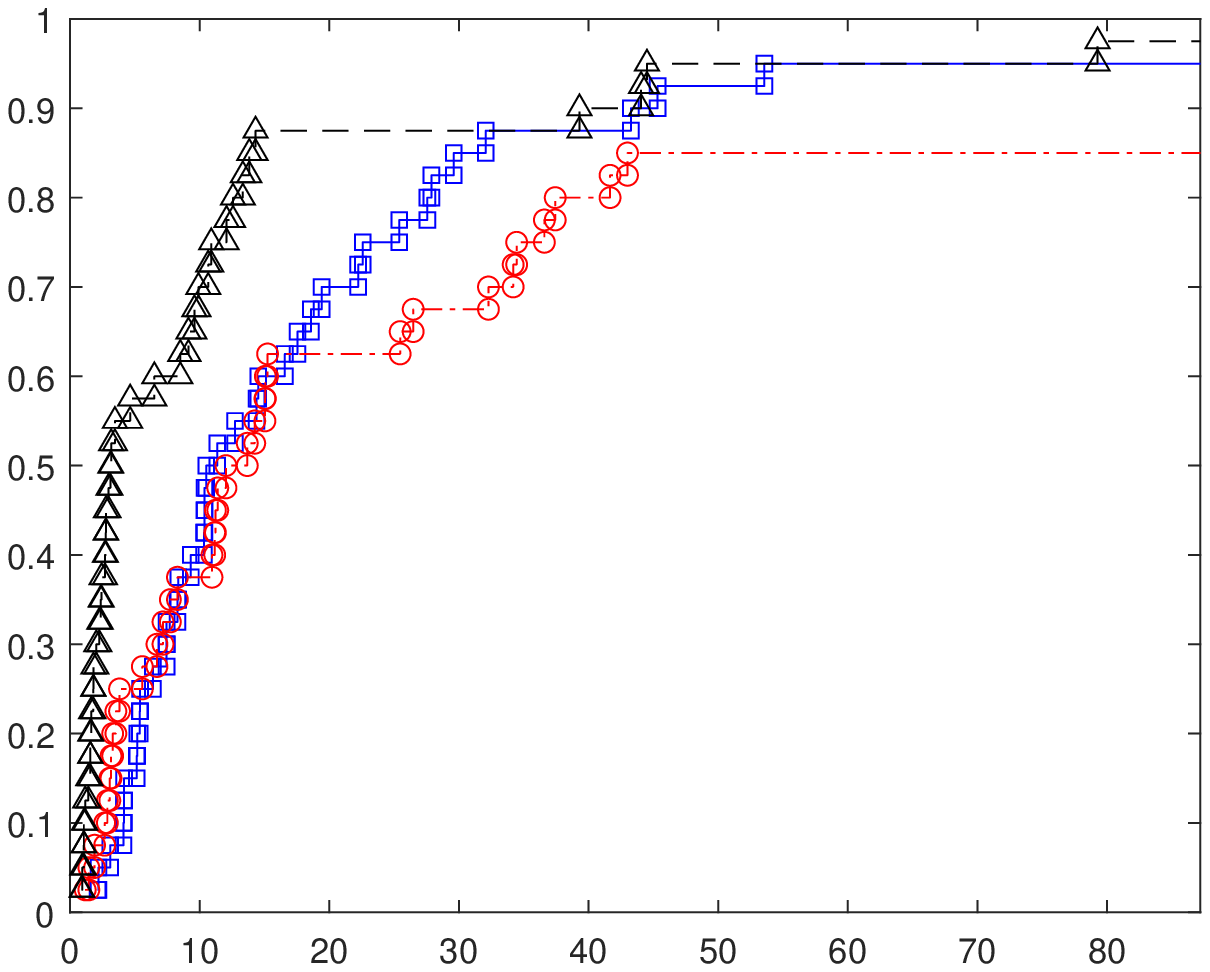}
	\caption{Data p., $\tau=10^{-1}$}
\end{subfigure}	
\begin{subfigure}[b]{0.21\textwidth}
	\includegraphics[width=\linewidth]{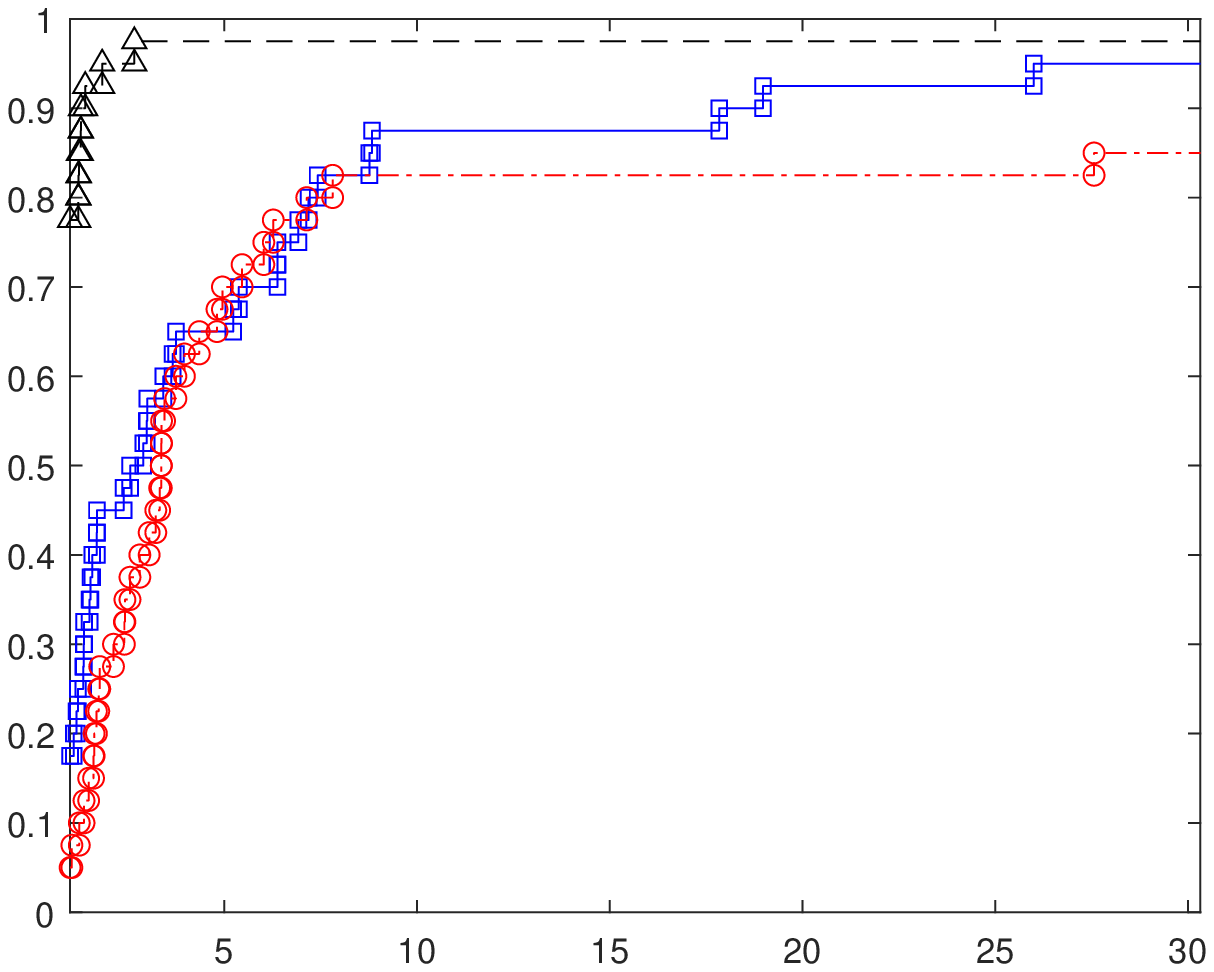}
	\caption{Perf. p., $\tau=10^{-1}$}
\end{subfigure}	
\begin{subfigure}[b]{0.21\textwidth}
	\includegraphics[width=\linewidth]{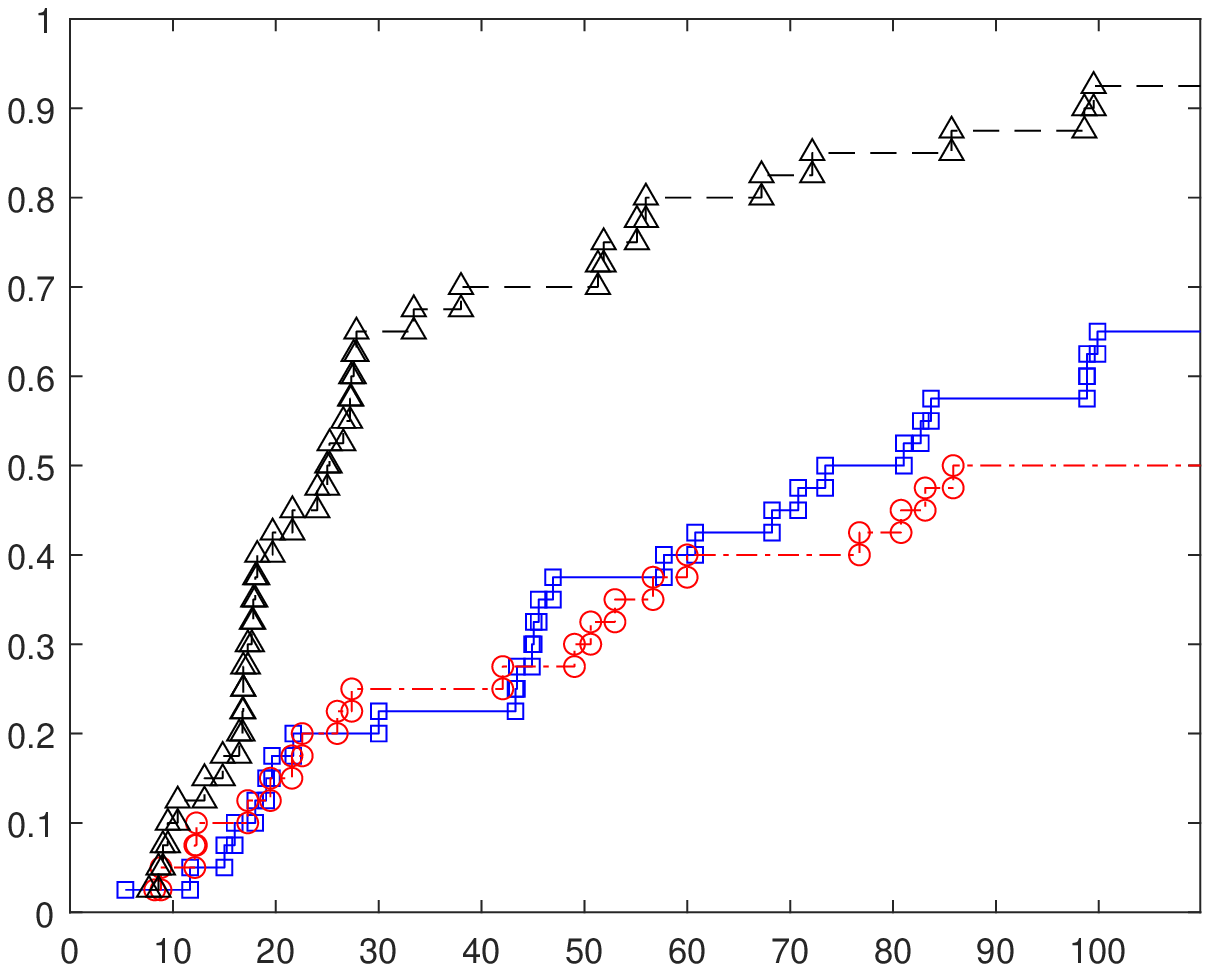}
	\caption{Data p., $\tau=10^{-3}$}
\end{subfigure}	
\begin{subfigure}[b]{0.21\textwidth}
	\includegraphics[width=\linewidth]{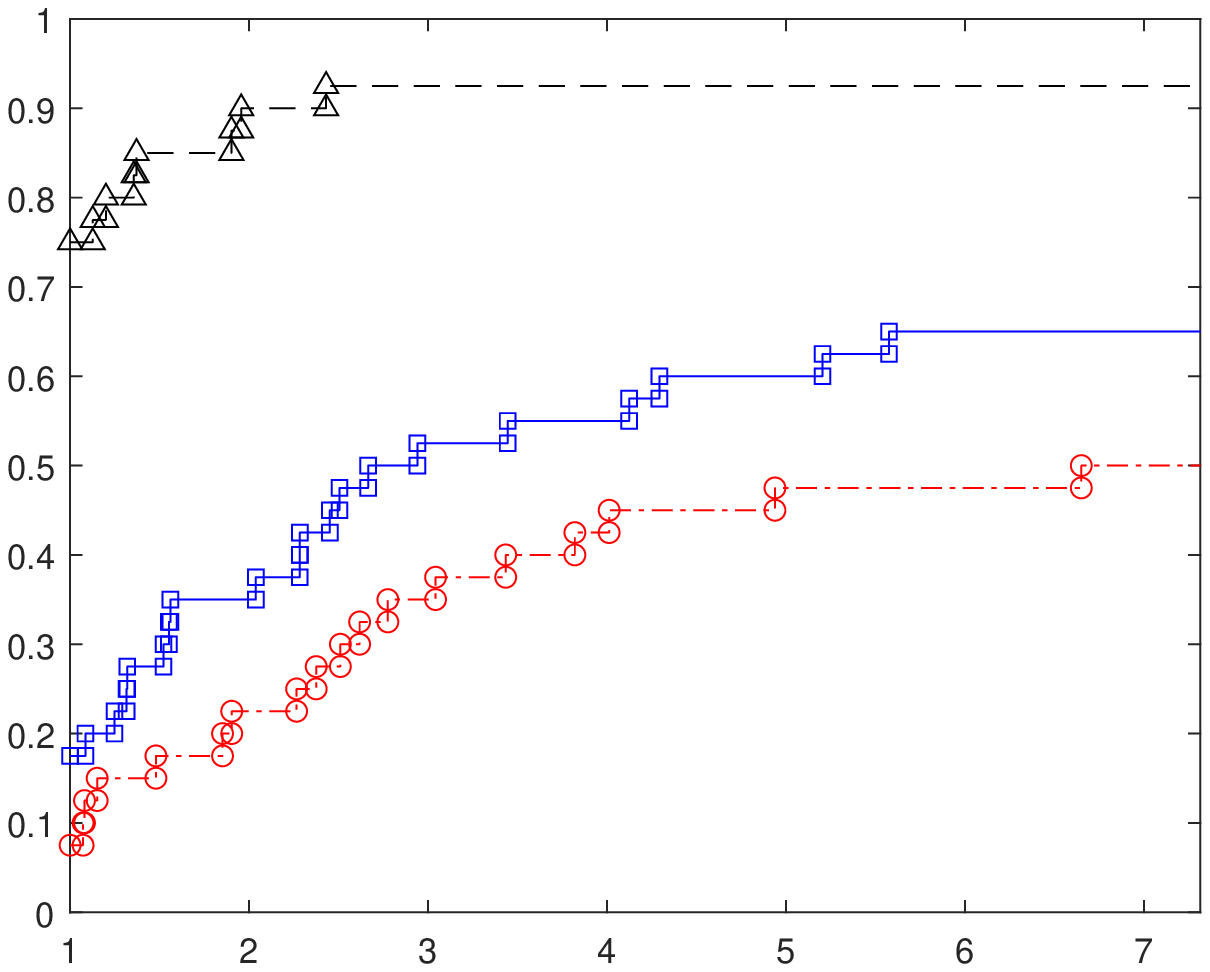}
	\caption{Perf. p., $\tau=10^{-3}$}
\end{subfigure}	
\vspace{3mm}

		\begin{subfigure}[b]{0.21\textwidth}
	\includegraphics[width=\linewidth]{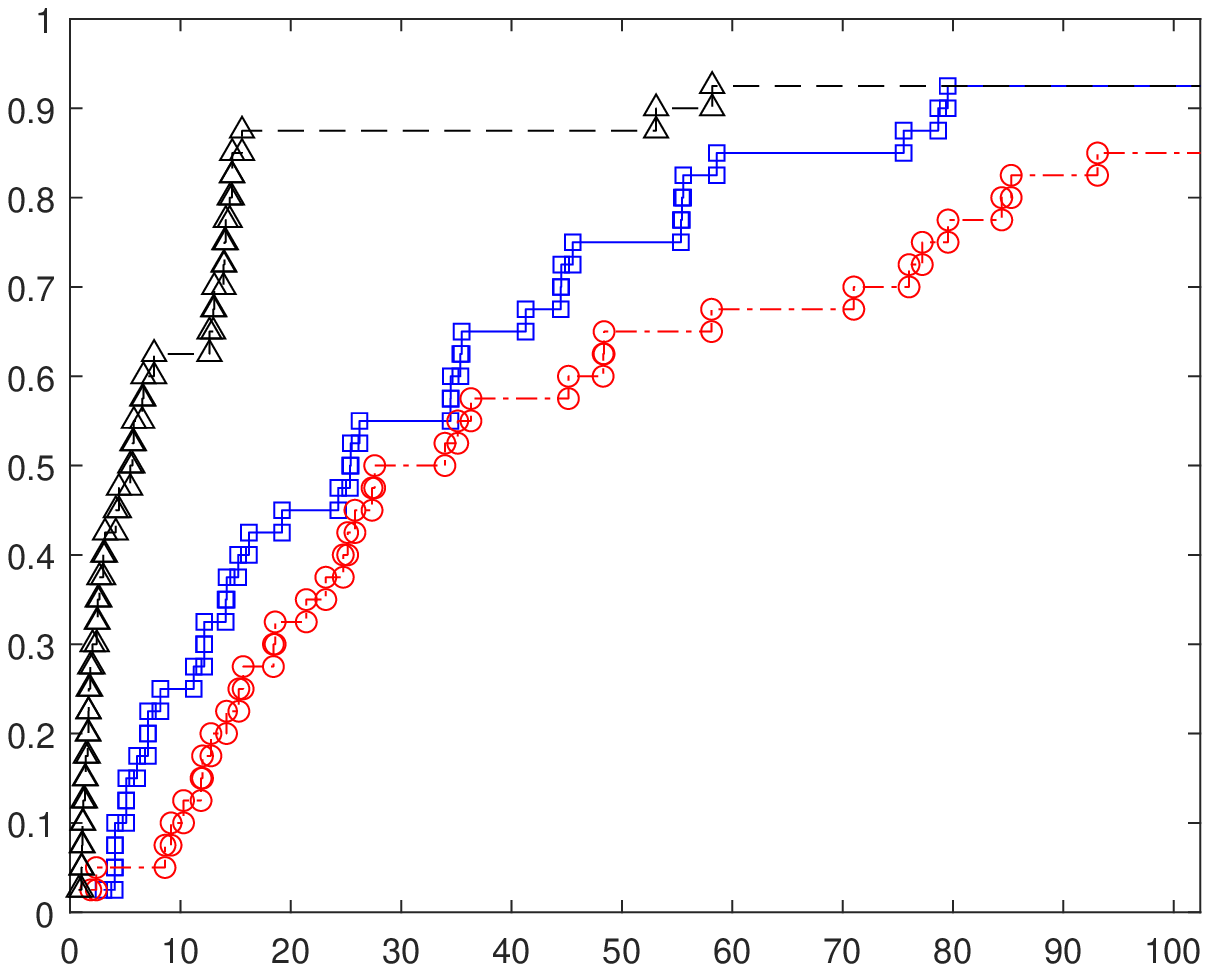}
	\caption{Data p., $\tau=10^{-1}$}
\end{subfigure}	
\begin{subfigure}[b]{0.21\textwidth}
	\includegraphics[width=\linewidth]{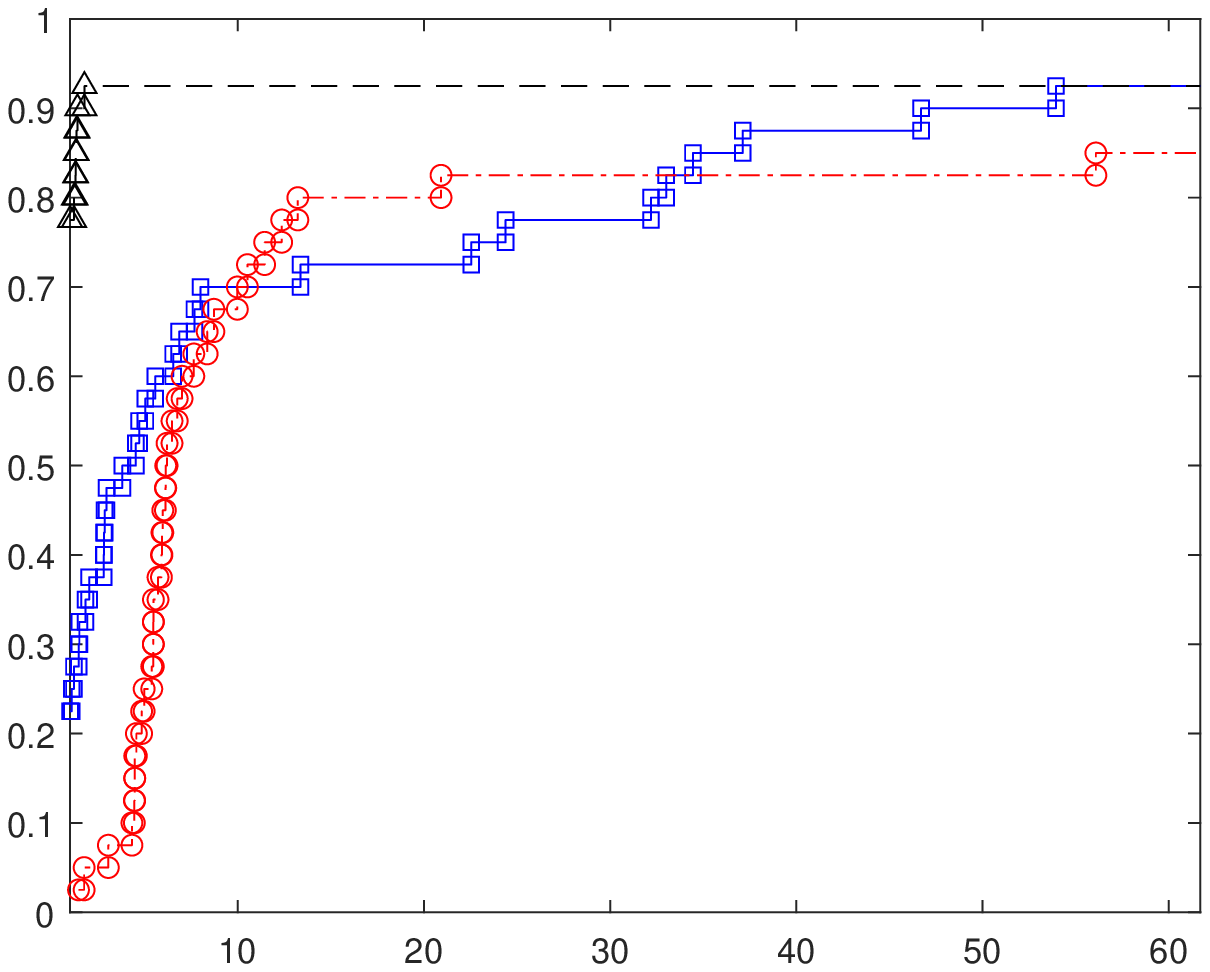}
	\caption{Perf. p., $\tau=10^{-1}$}
\end{subfigure}	
\begin{subfigure}[b]{0.21\textwidth}
	\includegraphics[width=\linewidth]{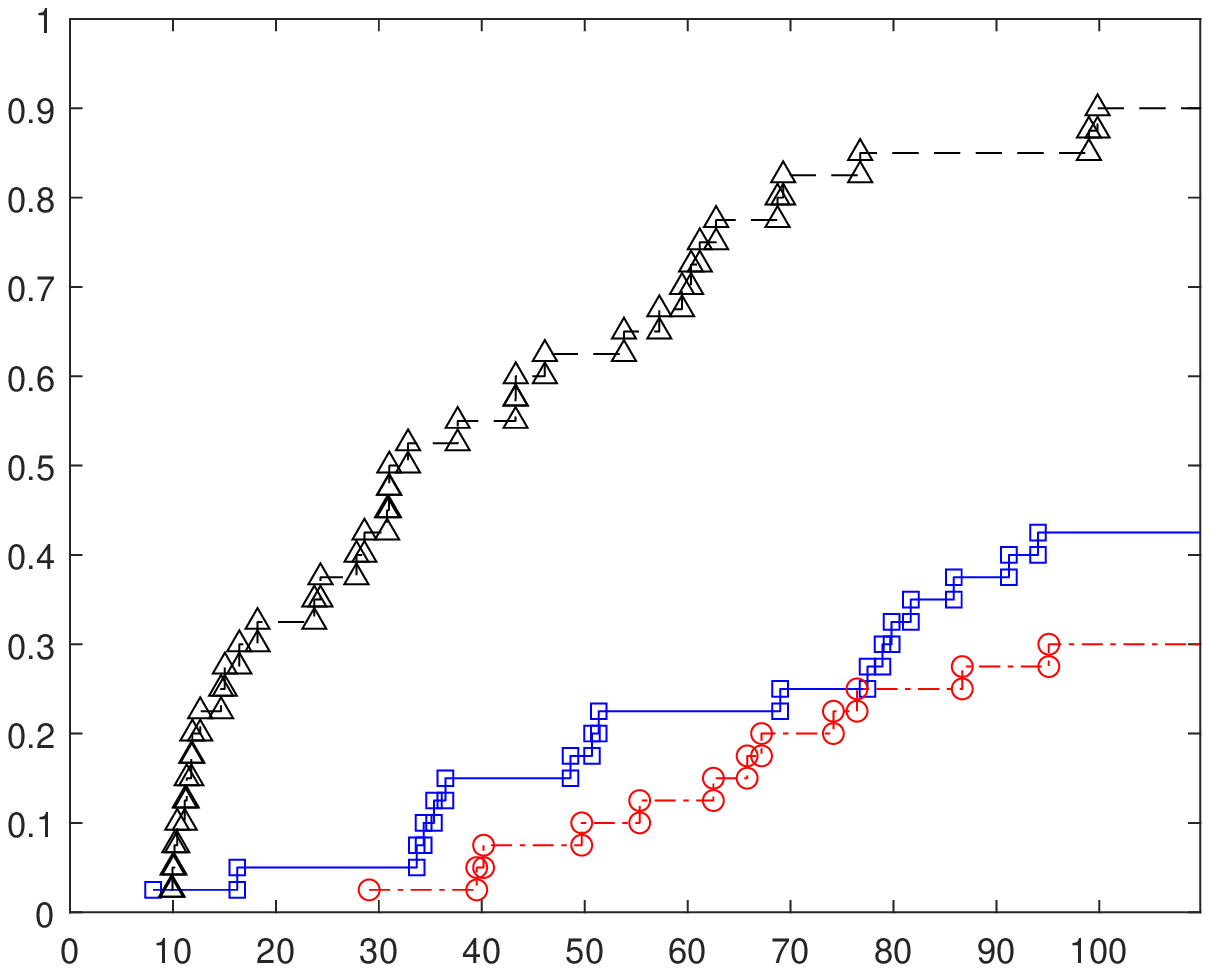}
	\caption{Data p., $\tau=10^{-3}$}
\end{subfigure}	
\begin{subfigure}[b]{0.21\textwidth}
	\includegraphics[width=\linewidth]{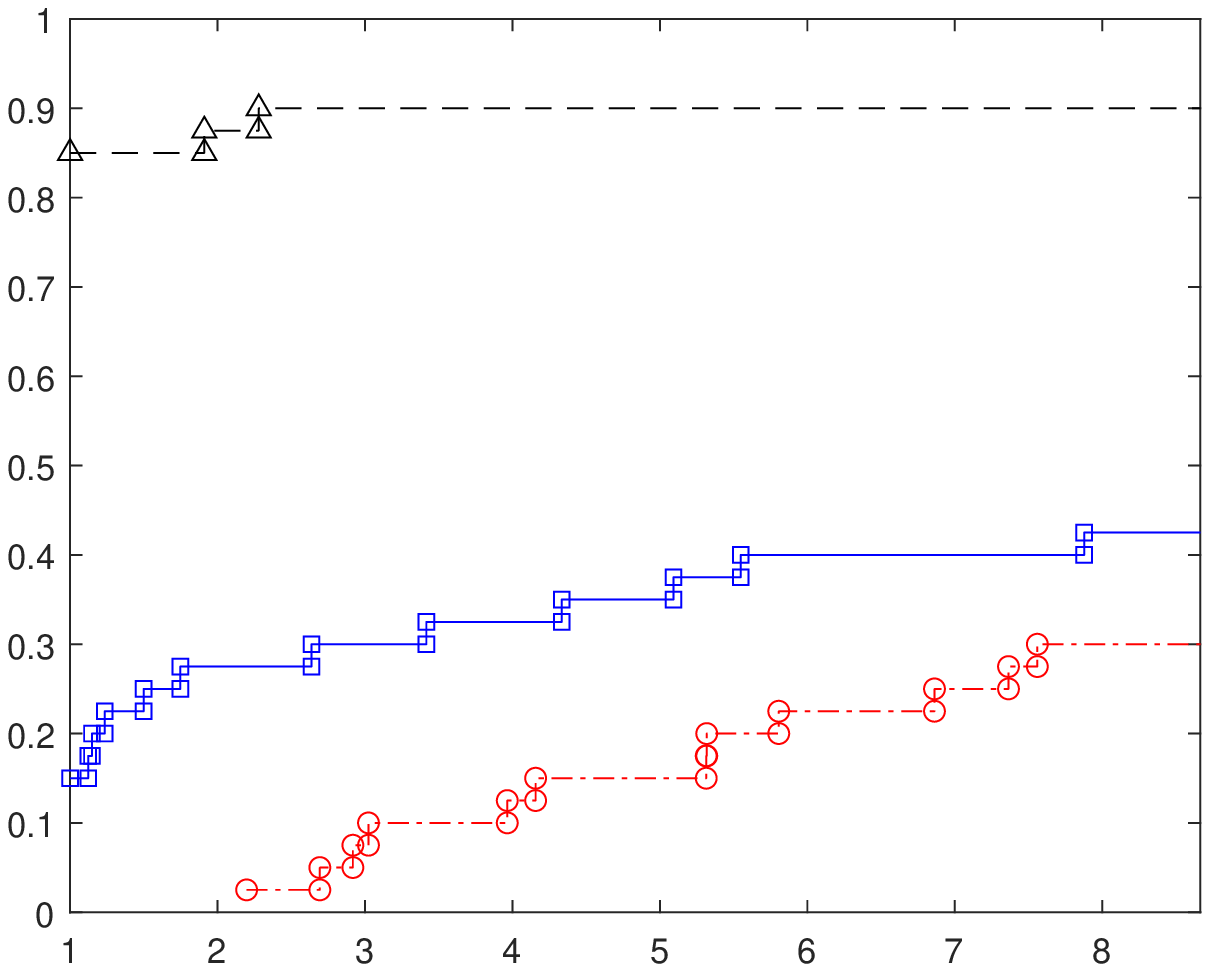}
	\caption{Perf. p., $\tau=10^{-3}$}
\end{subfigure}		
	\caption{From top to bottom: results for small, medium and large instances in the smooth case.}
\label{fig:t1s}
\end{figure}
\vspace{3mm}

	\begin{figure}[h]
	\centering
		\begin{subfigure}[b]{0.21\textwidth}
	\includegraphics[width=\linewidth]{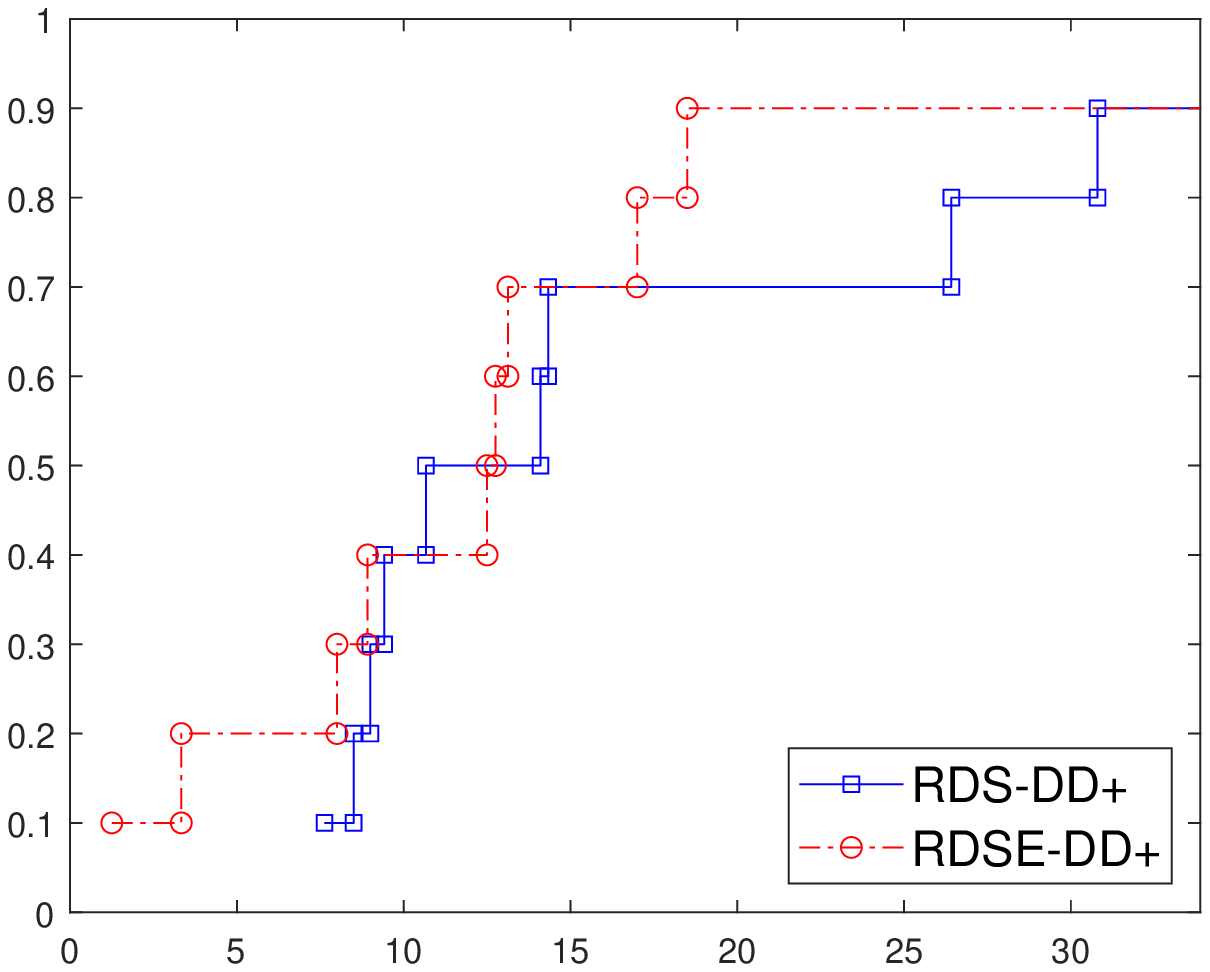}
	\caption{Data p., $\tau=10^{-1}$}
\end{subfigure}	
\begin{subfigure}[b]{0.21\textwidth}
	\includegraphics[width=\linewidth]{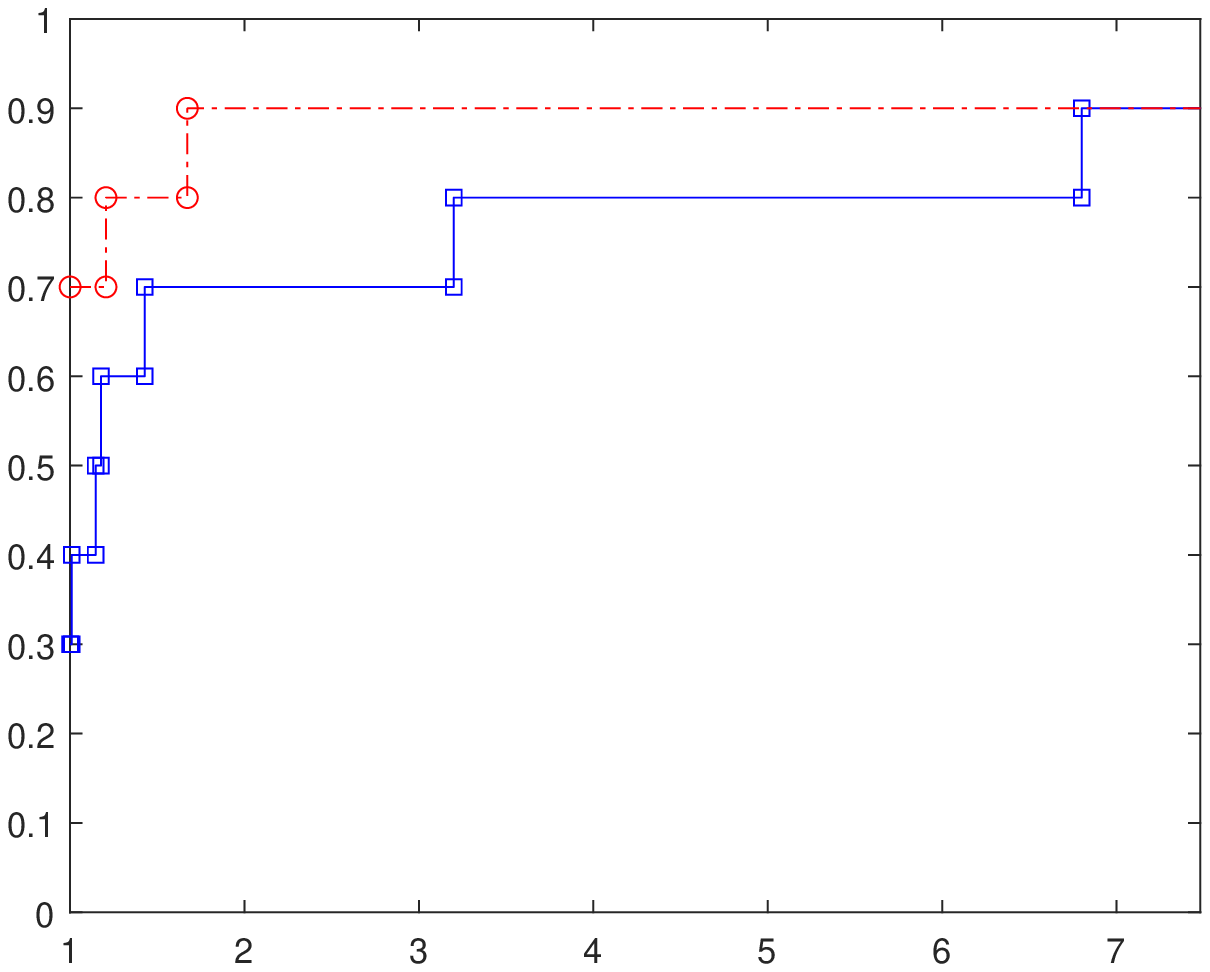}
	\caption{Perf. p., $\tau=10^{-1}$}
\end{subfigure}	
\begin{subfigure}[b]{0.21\textwidth}
	\includegraphics[width=\linewidth]{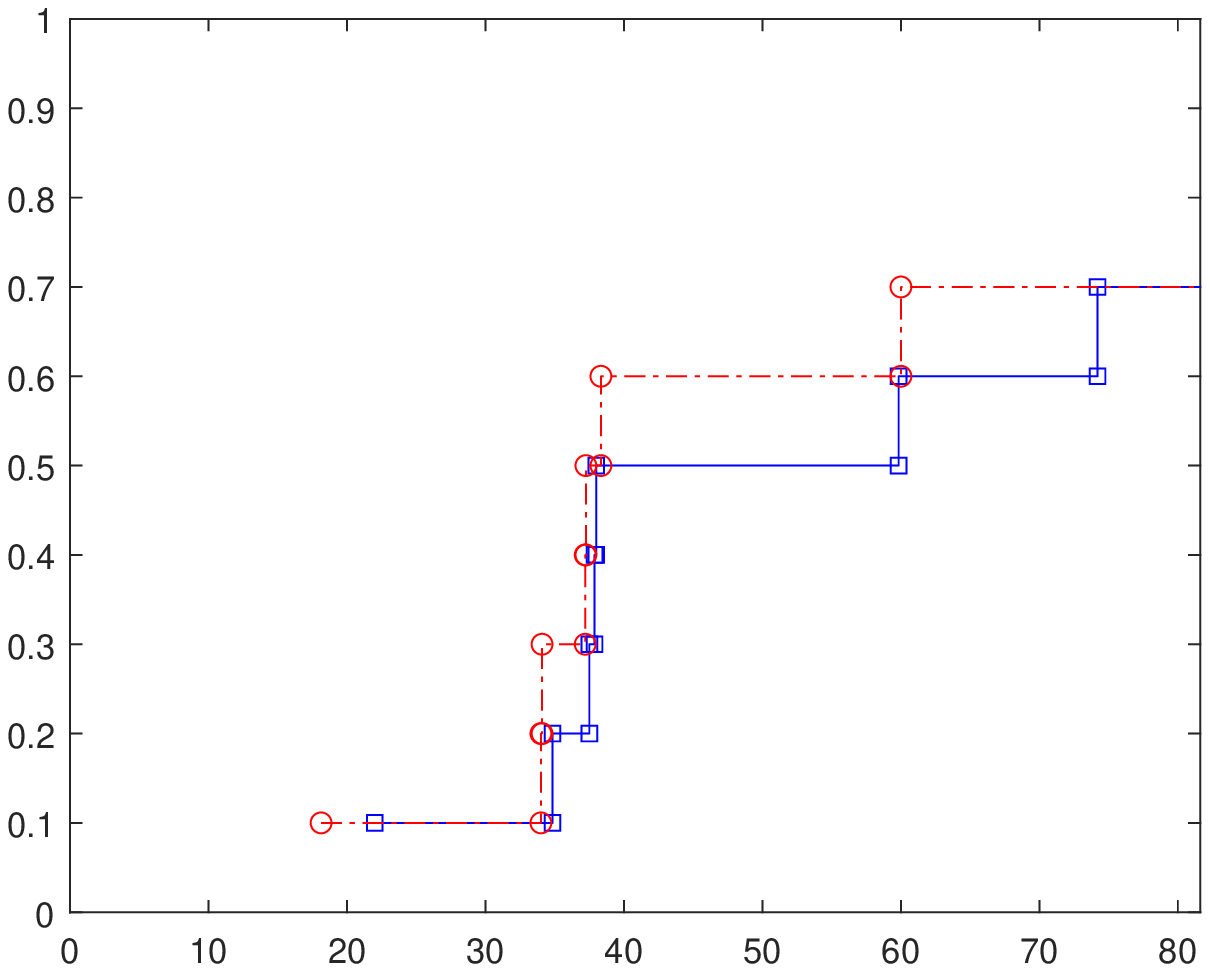}
	\caption{Data p., $\tau=10^{-3}$}
\end{subfigure}	
\begin{subfigure}[b]{0.21\textwidth}
	\includegraphics[width=\linewidth]{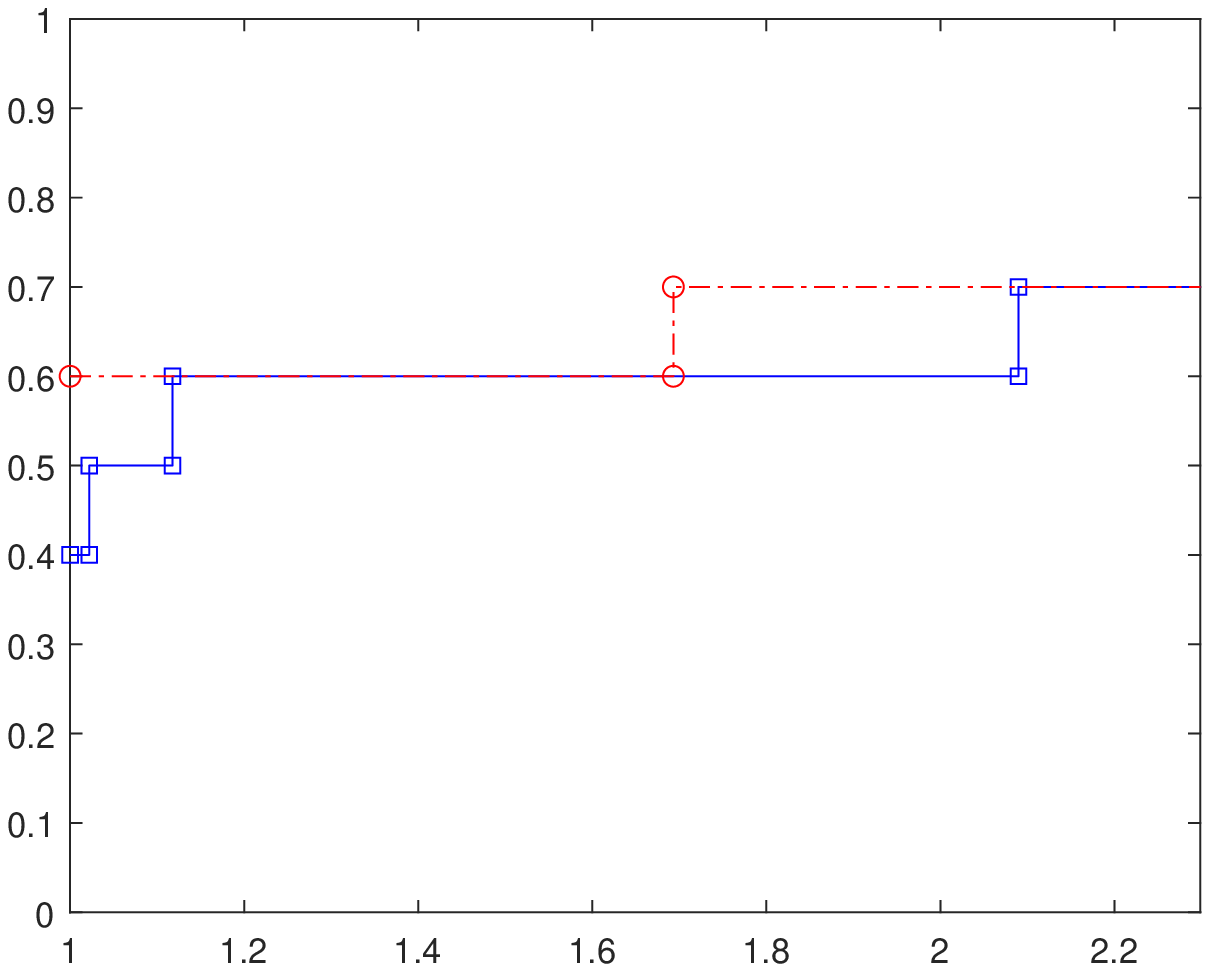}
	\caption{Perf. p., $\tau=10^{-3}$}
\end{subfigure}	
\vspace{3mm}

	\begin{subfigure}[b]{0.21\textwidth}
	\includegraphics[width=\linewidth]{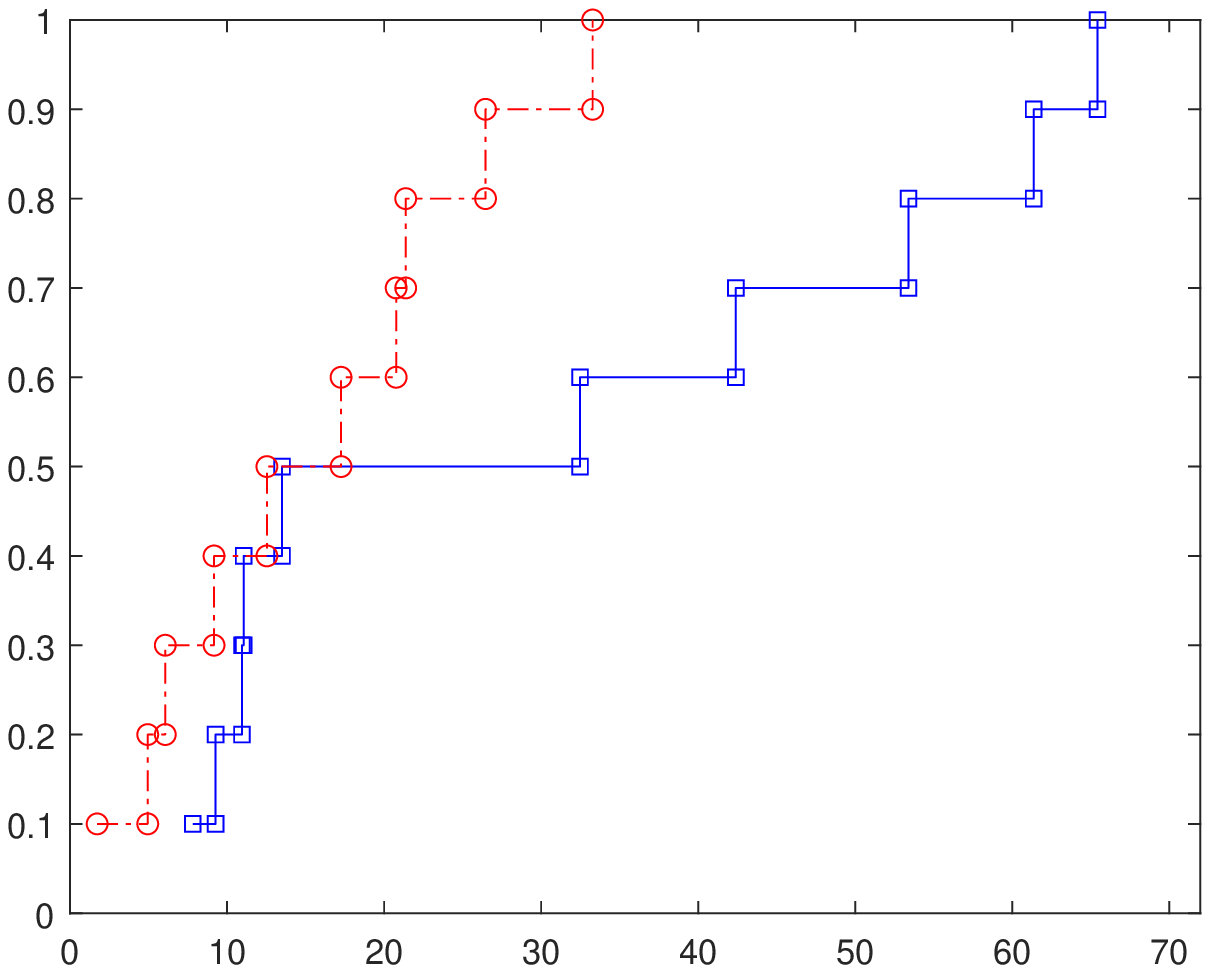}
	\caption{Data p., $\tau=10^{-1}$}
\end{subfigure}	
\begin{subfigure}[b]{0.21\textwidth}
	\includegraphics[width=\linewidth]{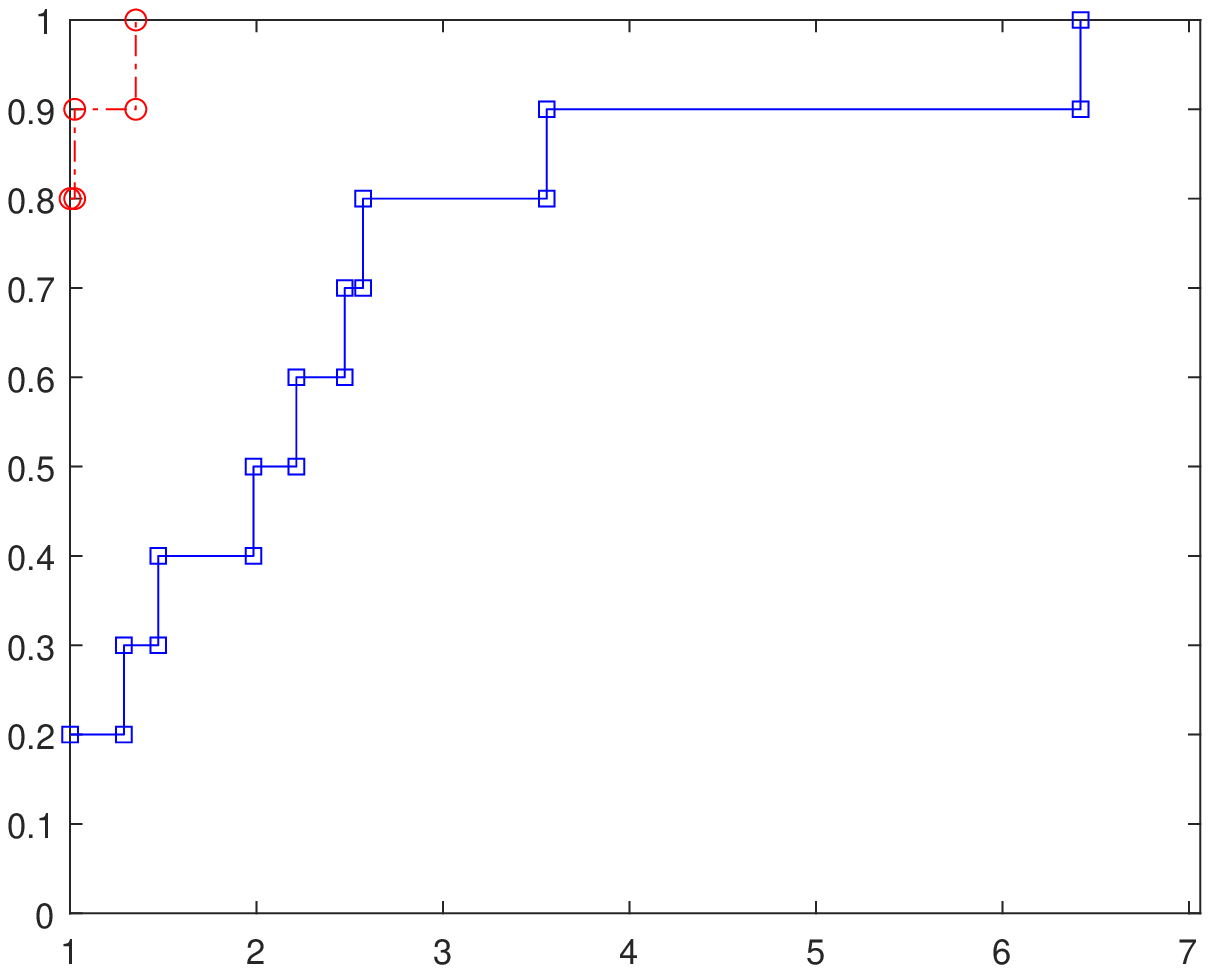}
	\caption{Perf. p., $\tau=10^{-1}$}
\end{subfigure}	
\begin{subfigure}[b]{0.21\textwidth}
	\includegraphics[width=\linewidth]{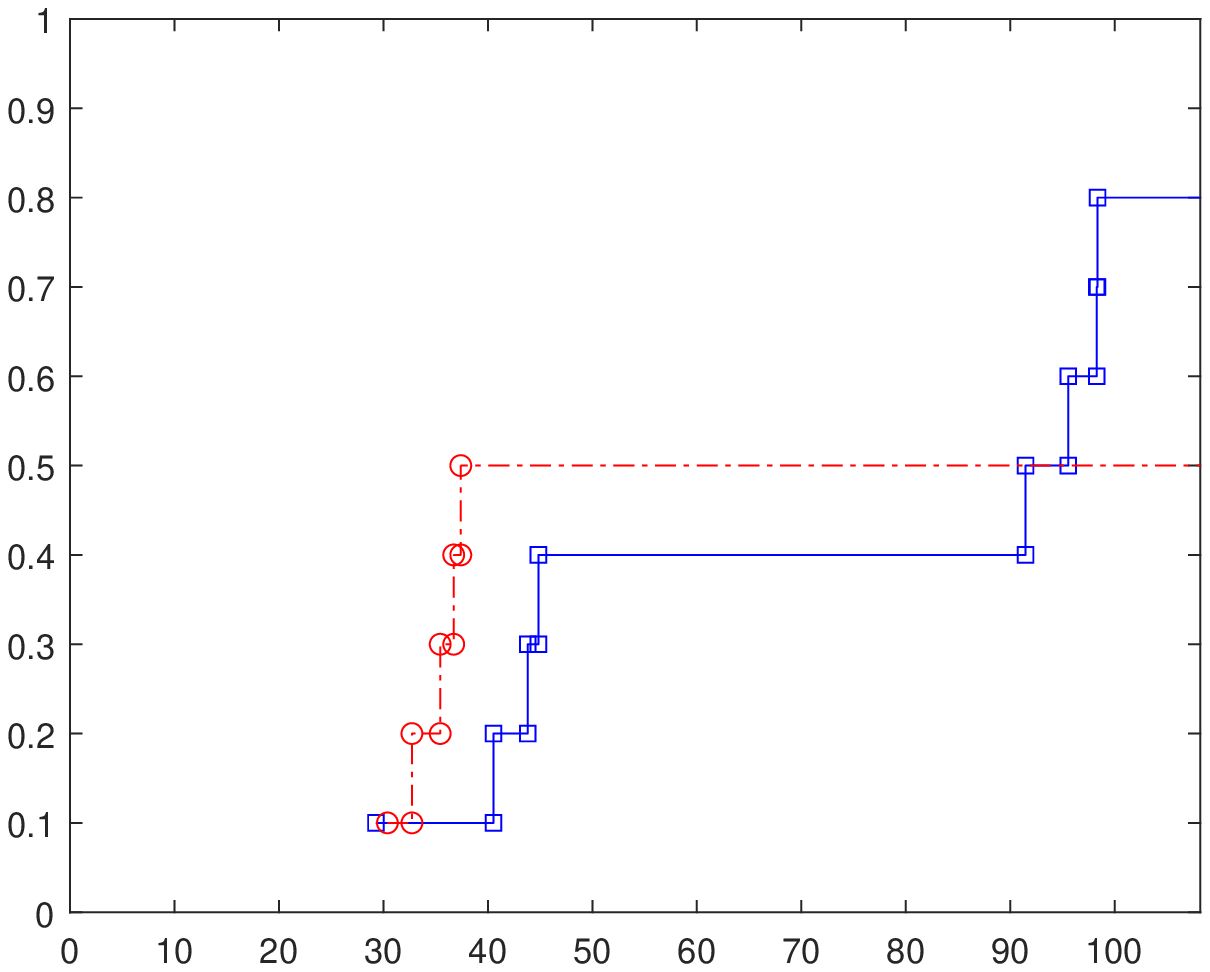}
	\caption{Data p., $\tau=10^{-3}$}
\end{subfigure}	
\begin{subfigure}[b]{0.21\textwidth}
	\includegraphics[width=\linewidth]{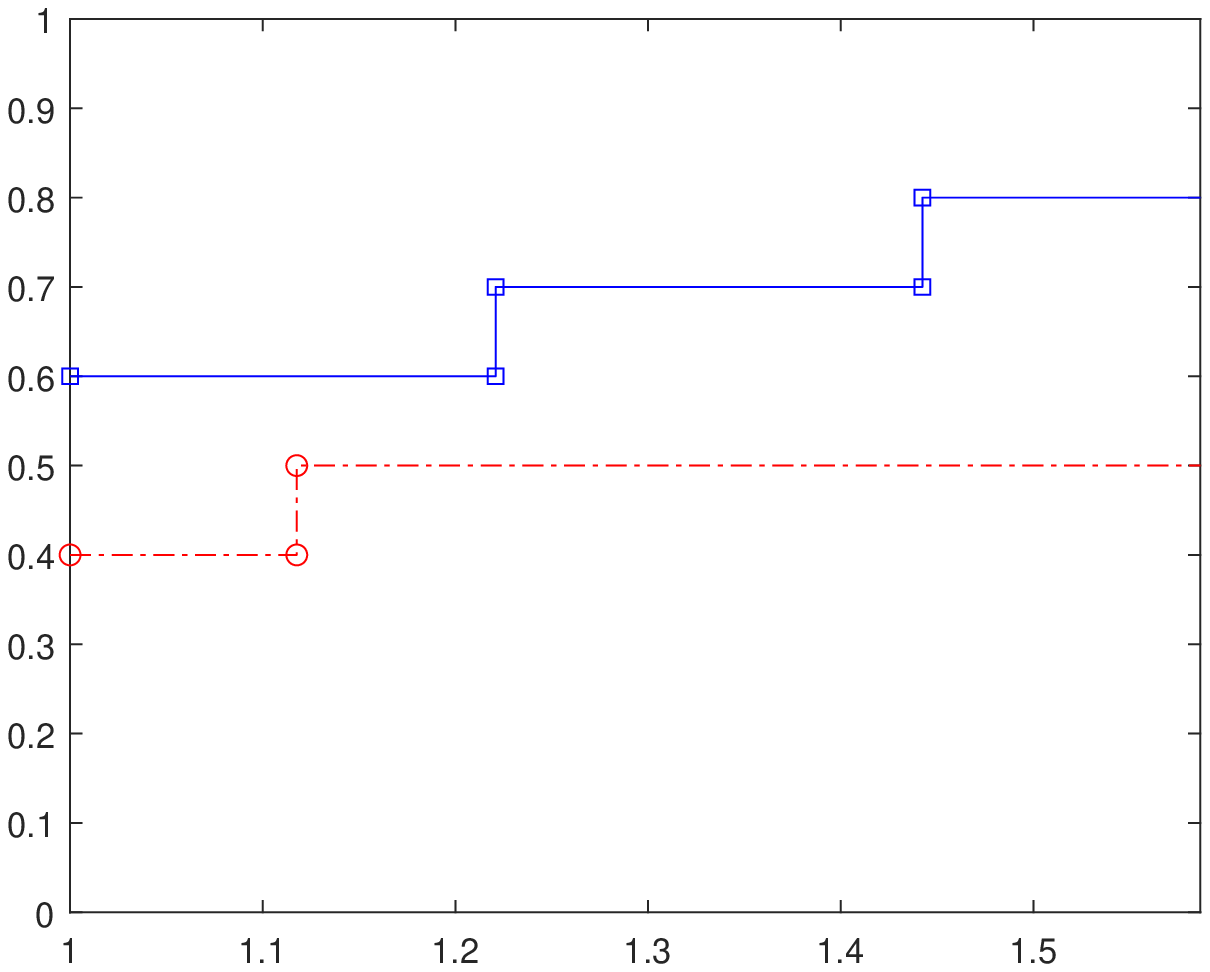}
	\caption{Perf. p., $\tau=10^{-3}$}
\end{subfigure}	
\vspace{3mm}

	\begin{subfigure}[b]{0.21\textwidth}
	\includegraphics[width=\linewidth]{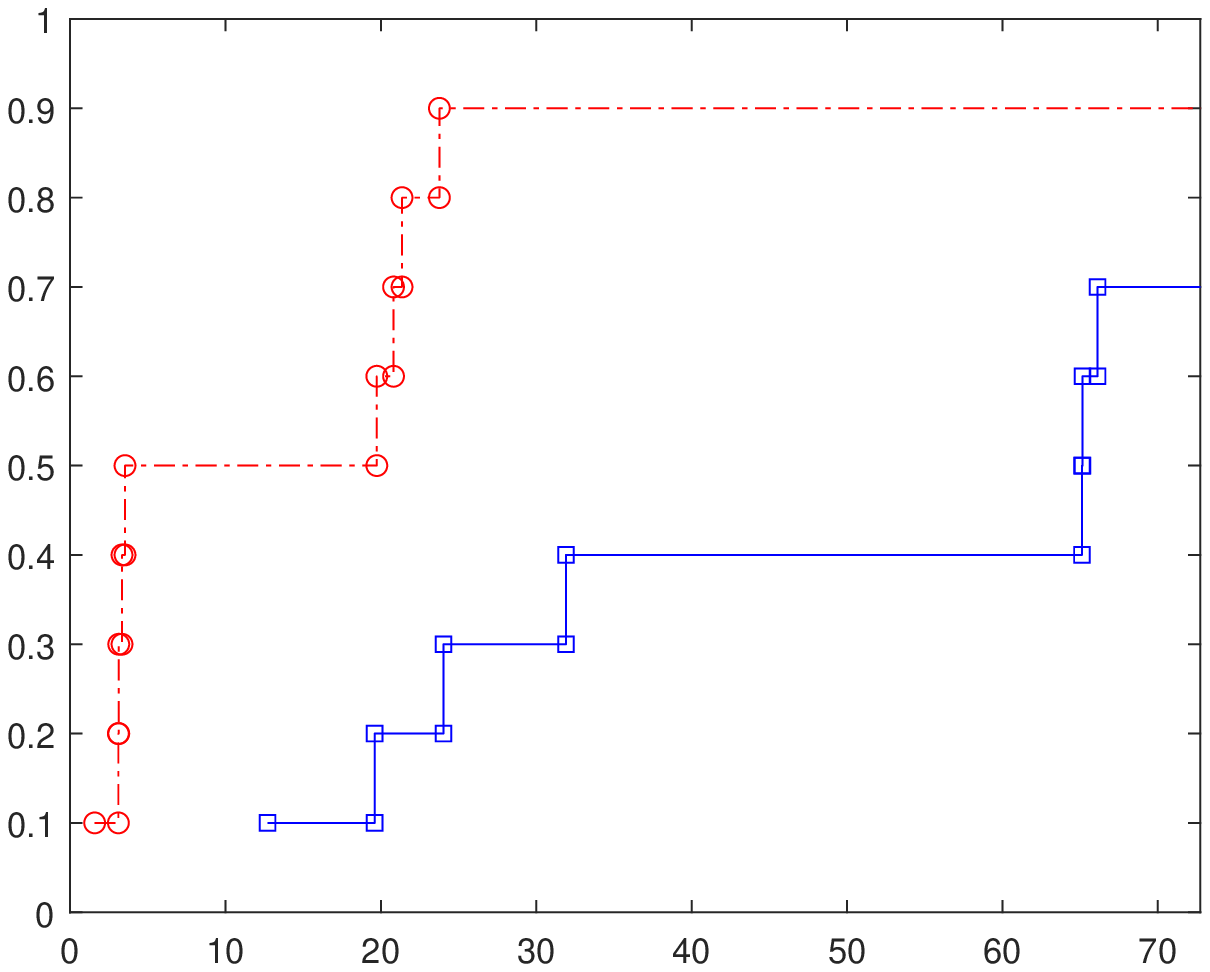}
	\caption{Data p., $\tau=10^{-1}$}
\end{subfigure}	
\begin{subfigure}[b]{0.21\textwidth}
	\includegraphics[width=\linewidth]{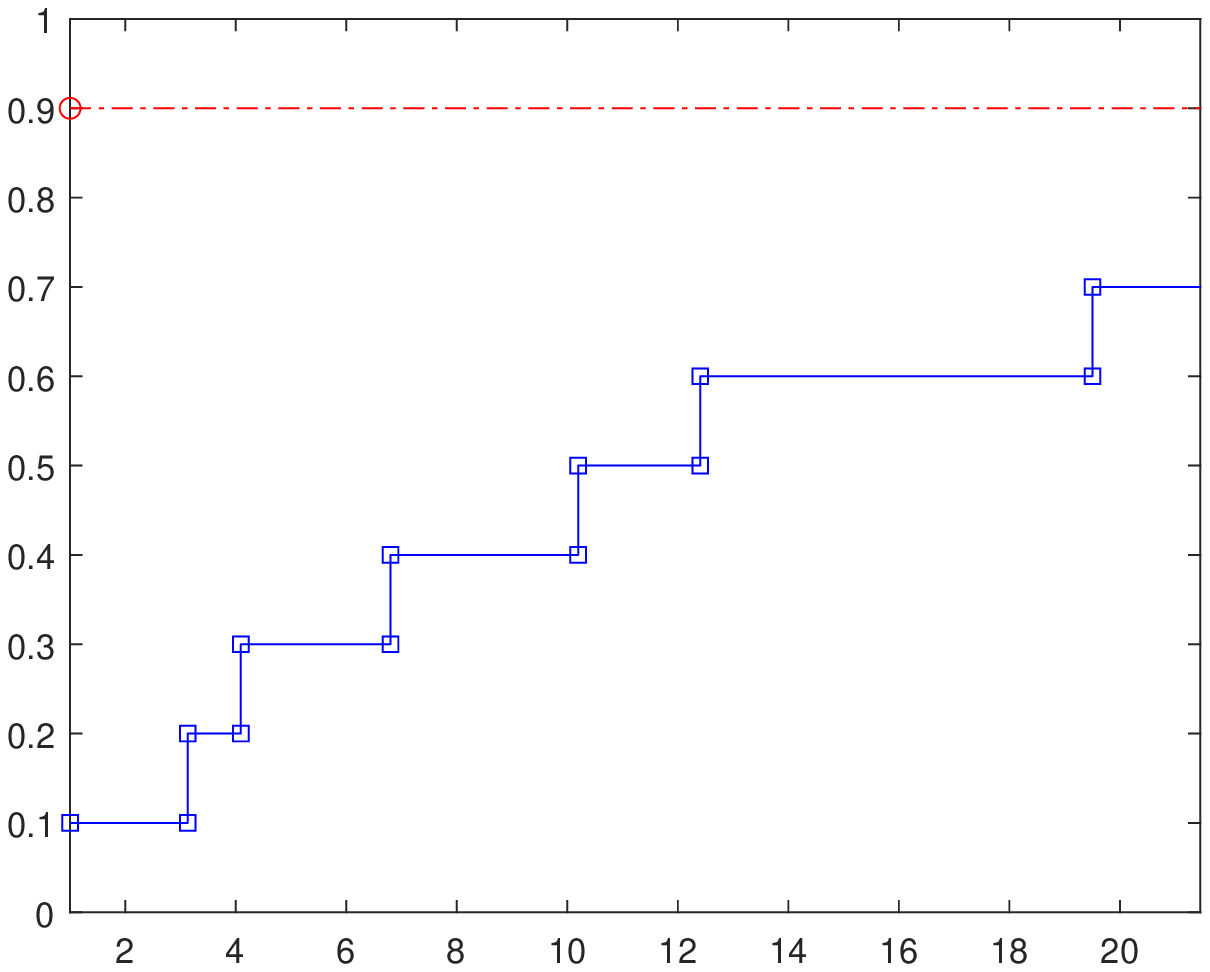}
	\caption{Perf. p., $\tau=10^{-1}$}
\end{subfigure}	
\begin{subfigure}[b]{0.21\textwidth}
	\includegraphics[width=\linewidth]{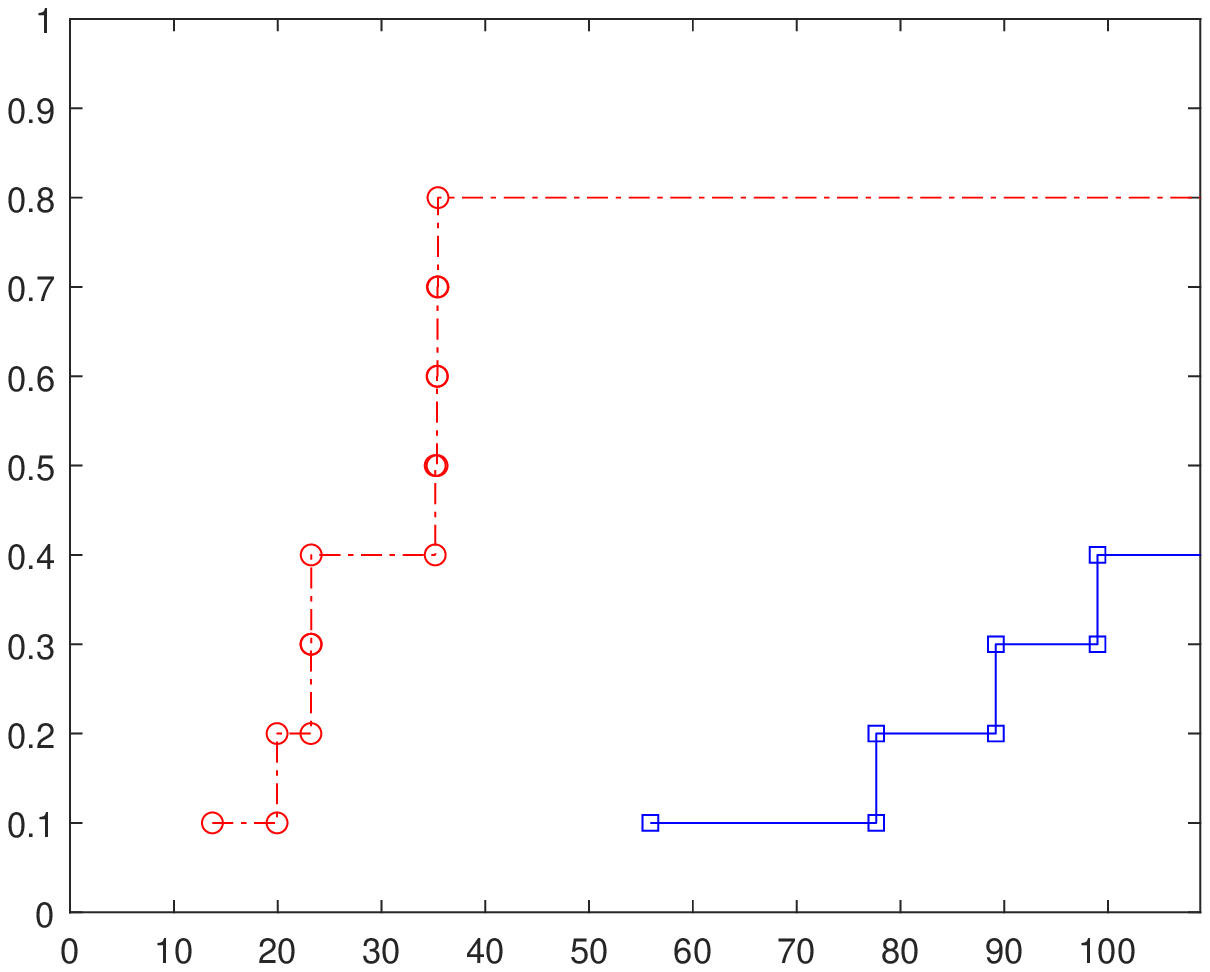}
	\caption{Data p., $\tau=10^{-3}$}
\end{subfigure}	
\begin{subfigure}[b]{0.21\textwidth}
	\includegraphics[width=\linewidth]{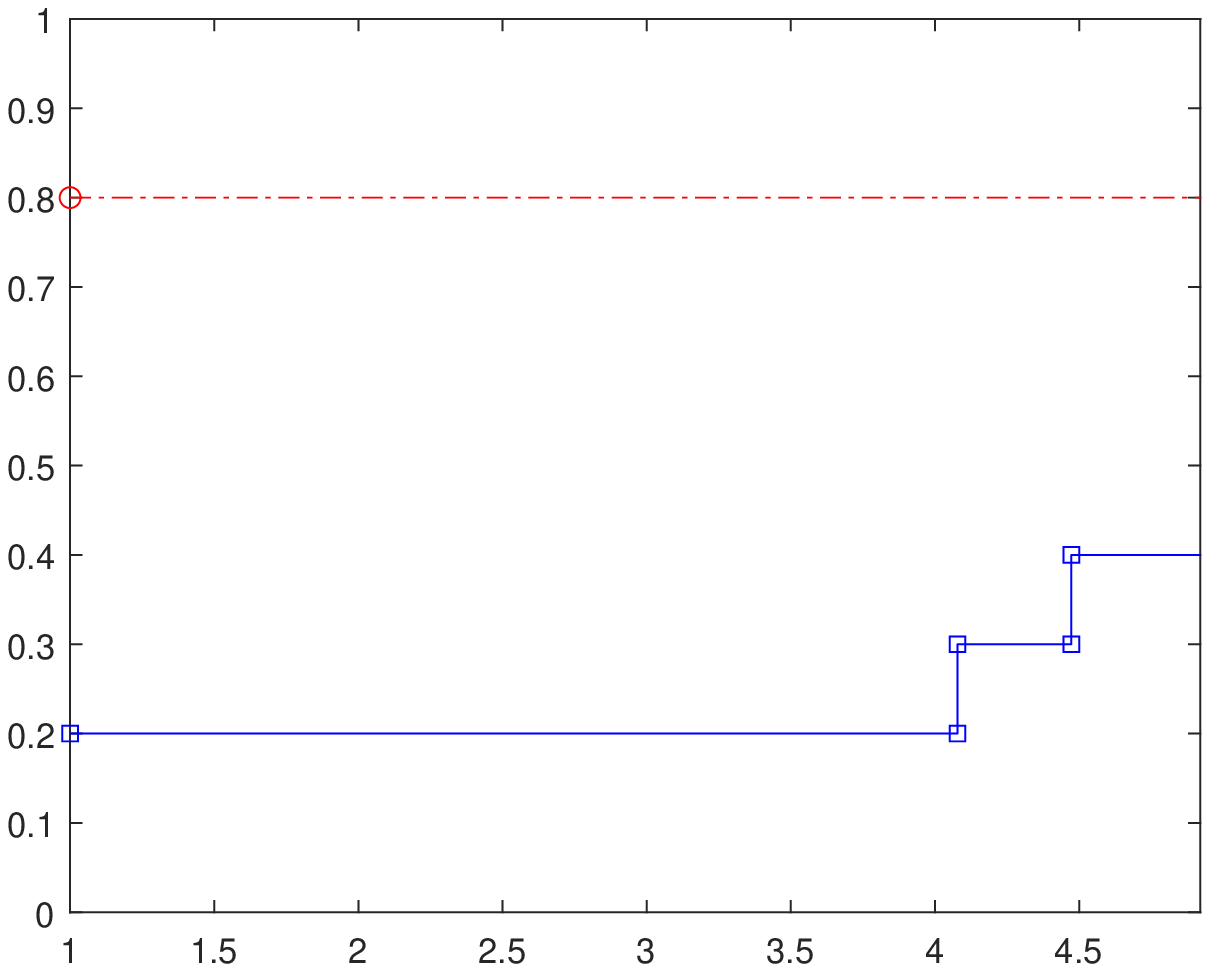}
	\caption{Perf. p., $\tau=10^{-3}$}
\end{subfigure}	
	\caption{From top to bottom: results for small, medium and large instances in the nonsmooth case.}
	\label{fig:t3s}
\end{figure}

\end{document}